\documentclass{article}

\usepackage{authblk}

\usepackage[utf8]{inputenc}
\usepackage[T1]{fontenc}
\usepackage[english]{babel}
\usepackage{textcomp}
\usepackage{amsmath,amssymb,amsthm}
\usepackage{lmodern}
\usepackage[a4paper]{geometry}
\usepackage{xcolor, pict2e}
\usepackage{microtype}
\usepackage{caption}
\usepackage{listings}
\usepackage{multicol}
\usepackage{moreverb}
\usepackage{hyperref}
\hypersetup{pdfstartview=XYZ}
\usepackage{wrapfig}
\usepackage[sans]{dsfont}

\usepackage{tikz, pgfplots}
\usetikzlibrary{positioning}
\usepackage{amsmath,amssymb}
\usetikzlibrary{decorations.pathreplacing}

\usepackage{ytableau}
\usepackage{mathdots}

\usepackage{hyperref}
\hypersetup{
    colorlinks=true,
    linkcolor=blue,
    citecolor=magenta,
    urlcolor=blue,
    pdfborder={0 0 0}
}

\newcommand{\Poiss}{\ensuremath{\textup{Poiss}}}
\newcommand{\dtv}{\ensuremath{\textup{d}_\textup{TV}}}
\newcommand{\RI}{\ensuremath{\textup{RI}}}
\newcommand{\TRI}{\ensuremath{\textup{TRI}}}

\newcommand{\Inv}{\ensuremath{\textup{Inv}}}
\newcommand{\tmix}{\ensuremath{\text{t}_{\textup{mix}}}}

\newcommand{\ch}{\ensuremath{\textup{ch}}}
\newcommand{\Id}{\ensuremath{\textup{Id}}}

\newcommand{\Unif}{\ensuremath{\textup{Unif}}}
\newcommand{\Bin}{\ensuremath{\textup{Bin}}}

\newcommand{\rsym}{\ensuremath{r}_\textup{sym}}

\newcommand{\kS}{{\ensuremath{\mathfrak{S}}} }
\newcommand{\kA}{{\ensuremath{\mathfrak{A}}} }
\newcommand{\kE}{{\ensuremath{\mathfrak{E}}} }

\newcommand{\kY}{{\ensuremath{\mathfrak{Y}}} }

\newcommand{\Tr}{{\ensuremath{\textup{Tr}}} }
\newcommand{\RT}{{\ensuremath{\textup{RT}}} }
\newcommand{\pRT}{{\ensuremath{\textup{pRT}}} }
\newcommand{\T}{{\ensuremath{\textup{T}}} }
\newcommand{\height}{{\ensuremath{\textup{ht}}} }
\newcommand{\supp}{{\ensuremath{\textup{supp}}} }
\newcommand{\Conj}{{\ensuremath{\textup{Conj}}} }

\newcommand{\sgn}{{\ensuremath{\textup{sgn}}}}
 
\usepackage[font=sf, labelfont={sf,bf}, margin=1cm]{caption}
\usepackage{enumitem}
\setlist[itemize,1]{nosep}
\setlist[enumerate,1]{itemsep=0pt,label=(\alph*)}

\usepackage{float}
\usepackage[caption = false]{subfig}
\usepackage{graphicx}

\newtheorem*{theorem*}{Theorem}
\newtheorem{theorem}{Theorem}[section]
\newtheorem{proposition}[theorem]{Proposition}
\newtheorem{lemma}[theorem]{Lemma}
\newtheorem{corollary}[theorem]{Corollary}

\theoremstyle{remark}
\newtheorem{remark}{Remark}[section]

\usepackage{todonotes}

\newcommand{\cC}{{\ensuremath{\mathcal C}} }

\newcommand{\cT}{{\ensuremath{\mathcal T}} }

\newcommand{\bbE}{{\ensuremath{\mathbb E}} }
\newcommand{\bbN}{{\ensuremath{\mathbb N}} }
\newcommand{\bbZ}{{\ensuremath{\mathbb Z}} }

\newcommand{\bbR}{{\ensuremath{\mathbb R}} }

\newcommand{\bbP}{{\ensuremath{\mathbb P}} }

\newcommand{\ag}{\left\{ }

\newcommand{\ad}{\right\} }

\newcommand{\cg}{\left[}
\newcommand{\cd}{\right]}
\newcommand{\pg}{\left(}
\newcommand{\pd}{\right)}
\newcommand{\bg}{\left|}
\newcommand{\bd}{\right|}
\newcommand{\lf}{\left\lfloor}
\newcommand{\rf}{\right\rfloor}
\newcommand{\lc}{\left\lceil}
\newcommand{\rc}{\right\rceil}

\newcommand{\du}{{\ensuremath{\;:\;}}}

\makeatletter
\newcommand*\bigcdot{\mathpalette\bigcdot@{.5}}
\newcommand*\bigcdot@[2]{\mathbin{\vcenter{\hbox{\scalebox{#2}{$\m@th#1\bullet$}}}}}
\makeatother

\numberwithin{equation}{section}

\pgfplotsset{compat=1.18}

\begin{document}

\renewcommand{\theparagraph}{\thesubsection.\arabic{paragraph}} 
\title{Cutoff profiles for conjugacy invariant random walks on symmetric groups}

\author{Lucas Teyssier}
\affil{Université de Lorraine, \texttt{lucas.teyssier@univ-lorraine.fr}}

\maketitle
\begin{abstract}
We prove asymptotic equivalents of characters for finite-level representations of symmetric groups, that is, for Young diagrams which have all but finitely many boxes on their first row. The proofs rely on computing the number of ribbon tableaux of different types, which allows us to estimate characters via the Murnaghan--Nakayama rule. 

We deduce that random walks on symmetric groups generated by conjugacy classes with a macroscopic number of fixed points have a Poissonian cutoff profile. We also prove that the random involution walk exhibits cutoff and find its cutoff profile. Finally, we obtain numerics for the random transposition walk on a deck of 52 cards, giving concrete estimates on the question that originally motivated Diaconis and Shahshahani.
\end{abstract}

\tableofcontents

\section{Introduction}
\subsection{Some history}
Let $n\geq 2$ and let $\kS_n$ be the symmetric group of index $n$. Given a probability measure $\mu$ on $\kS_n$, we refer to the discrete time Markov chain with transition probabilities $P(\sigma, \tau) = \mu(\sigma^{-1}\tau)$ for $\sigma, \tau \in \kS_n$ as the \textit{random walk driven by $\mu$} or \textit{$\mu$-walk}.
Denote the uniform measure on a finite set $S$ by $\Unif_S$, and the conjugacy class of transpositions of $\kS_n$ by $\cT$.

The $\Unif_{\cT}$-walk (with laziness $1/n$) is a historically important model known as the random transposition shuffle.
In terms of card shuffling, this corresponds to picking at each step two cards independently and uniformly at random, and swapping them. The initial problem that emerged at the Bell Laboratories in the 1970's was to understand how many random transpositions are needed to mix a deck of 52 cards, and we refer to the conversation between Aldous and Diaconis \cite{Aldous2013AnotherConversationWithPersiDiaconis} for a more detailed history of the problem.

After initial attempts relying purely on probability failed, Diaconis and Shahshahani \cite{DiaconisShahshahani1981} solved this problem (asymptotically) using representation theory. They proved that a phase transition, called cutoff, occurs around $t_n := \frac{1}{2}n\ln n$ steps. Essentially, what they proved is that for any $\varepsilon\in (0,1)$, denoting the distribution of the walk after $t$ steps by $\mu_{t}$, we have
\begin{equation}\label{eq: Diaconis--Shahshahani cutoff}
        \dtv\pg \mu_{\lf (1-\varepsilon) t_n\rf}, \Unif_{\kS_n}\pd \xrightarrow[n\to\infty]{} 1  \quad \text{ and } \quad   \dtv\pg \mu_{\lc (1+\varepsilon) t_n\rc}, \Unif_{\kS_n}\pd \xrightarrow[n\to\infty]{}  0,
    \end{equation}
where the total variation distance between two probability measures $\mu$ and $\nu$ on a finite set $S$ is given by $\dtv(\mu,\nu) = \frac{1}{2}\sum_{x\in S}|\mu(x) - \nu(x)|$.

This led to several refinements and generalizations, and in particular to $\Unif_{\cC}$-walks, where $\cC$ is a conjugacy class:
\cite{Roussel2000, BerestyckiSchrammZeitouni2011} 
for conjugacy classes with small support,
\cite{Hough2016, BerestyckiSengul2019} for conjugacy classes of support size $o(n)$, \cite{Lulovthesis1996, LarsenShalev2008, TeyssierThevenin2025virtualdegreeswitten} for fixed point free permutations, \cite{Roichman1996, MullerSchlagePuchta2007precutoff} for general $L^2$ bounds (that are however weaker than cutoff). Recently, in a companion paper, we proved the $L^2$ cutoff for all conjugacy classes with $\Theta(n)$ fixed points \cite{Olesker-TaylorTeyssierThevenin2025sharpboundsandcutoff}. The cutoff result of Diaconis and Shahshahani was also generalized to several models of non-uniform transpositions \cite{Diaconis1988StFlour,Lacoin2016cutoffadjacent, BateConnorMatheau-Raven2021,NestoridiYan2025BiasedTranspositions, ArfaeeNestoridi2025JucysMurphyTranspositions}. 

\medskip

What happens within the phase transition is also well understood for transpositions. In \cite{Teyssier2020}, we found the \textit{cutoff profile}: for any $a\in \bbR$, we have
\begin{equation}\label{eq:convergence profile transpositions}
    \dtv\pg \mu_{\lf t_n(1+a/\ln n)\rf} , \Unif_{\kS_n}\pd \xrightarrow[n\to \infty]{} \dtv\pg \Poiss(1), \Poiss\pg 1+e^{-a} \pd \pd,
\end{equation}
where $\Poiss(\alpha)$ denotes the Poisson law of parameter $\alpha>0$.
The convergence in \eqref{eq:convergence profile transpositions} was extended to cycles of length $o(n)$ in \cite{NestoridiOlesker-Taylor2022limitprofiles}, and to models that are not conjugacy invariant, such as star transpositions \cite{Nestoridi2024comparisonstar}. Intuitively, the profile involves Poisson laws because the last \textit{observable} to be mixed is the number of fixed points. 
This intuition was recently made rigorous for transpositions  \cite{JainSawhney2026transpositionprofileotherproof}, and generalized to decks with repeated cards \cite{Chen2026kcyclesrepeatedcards}. The separation profile for transpositions (and for other models) was also recently found \cite{FrancisNestoridi2026separation}.

Cutoff profiles have also been found in several other contexts: for instance for the hypercube \cite{DiaconisGrahamMorrison1990hypercube}, the riffle shuffle \cite{BayerDiaconis1992}, repeated averages \cite{ChatterjeeDiaconisSlyZhang2022repeatedaverages}, the symmetric and asymmetric simple exclusion processes \cite{Lacoin2016cutoffprofileexclusioncircle, BufetovNejjar2022}. Many other cutoff profiles have been found in the last few years. Recent accounts of the literature on cutoff profiles can be found in \cite{Olesker-TaylorSchmid2026ProfileBernoulliLaplace, Teyssier2026EveryCutoffProfileIsPossible}.

\subsection{Main results}

The main novelty of the paper consists of Theorems~\ref{thm: character bound finite level general with M sigma} and~\ref{thm: character estimate many fixed points with error term}, which give precise asymptotics for finite-level characters of symmetric groups. Our main application is Theorem~\ref{thm: profil TV pour les CC}, which finds the profile for conjugacy classes whose cycle structure has $\Theta(n)$ fixed points. We defer representation theoretic statements (and in particular those of Theorems~\ref{thm: character bound finite level general with M sigma} and~\ref{thm: character estimate many fixed points with error term}) to Section~\ref{s: Asymptotics of low-level characters}.

\medskip

Given a probability measure $\mu$ on $\kS_n$, the distribution of the $\mu$-walk (started at the identity permutation $\Id$) after $t$ steps is given by the $t$-fold convolution product $\mu^{*t}$, where the convolution product of two functions $f,g \du \kS_n \to \bbR$ is given by $(f \ast g)(\sigma)  = \sum_{\tau \in \kS_n} f(\tau)g(\tau^{-1}\sigma)$ for $\sigma \in \kS_n$.

Denote the set of conjugacy classes of $\kS_n$ by $\Conj(\kS_n)$, and set $\Conj^*(\kS_n) = \Conj(\kS_n)\backslash \ag \ag \Id \ad \ad$.
If $\cC\in \Conj(\kS_n)$, $\sigma \in \cC$, and $F$ is a class function, we may write $F(\cC)$ for $F(\sigma)$; for instance $\sgn(\cC)$ and $\sgn(\sigma)$ both denote the sign of the permutation $\sigma$, and for $i\geq 1$,
$f_i(\cC)$ and $f_i(\sigma)$ denote the number of $i$-cycles of $\sigma$.
For $\cC\in \Conj^*(\kS_n)$ such that $1\leq f_1(\cC)\leq n-1$,
we also set
\begin{equation}
    t_{\cC} = \frac{\ln n}{\ln(n/f_1(\cC))}.
\end{equation}

Given a conjugacy class $\cC$ of $\kS_n$ and an integer $t\geq 0$, we set $\kE(\cC,t) = \kS_n \backslash\kA_n$ if both $t$ and $\cC$ are odd (meaning for the second one that the sign of $\cC$ is $-1$), and $\kE(\cC,t) = \kA_n$ otherwise.

In \cite{Olesker-TaylorTeyssierThevenin2025sharpboundsandcutoff}, we proved that random walks associated to conjugacy classes $\cC$ with $\Theta(n)$ fixed points have an $L^2$ cutoff at $t_\cC$. 
The next statement shows that they furthermore have a Poissonian cutoff profile, as conjectured (for conjugacy classes with $o(n)$ fixed points) in \cite[Section~1.2]{Teyssier2020} and \cite[Conjecture~6.1]{Nestoridi2024comparisonstar}.

\begin{theorem}\label{thm: profil TV pour les CC}
Let $\delta>0$.
For each $n$, let $ \cC^{(n)} \in\Conj^*(\kS_n)$. Assume that $f_1(\cC^{(n)})\geq \delta n$ for $n$ large enough. Let  $a\in \bbR$.
Let $S\subset \bbN^*$ and $(t_n)_{n\in S}$ be a sequence of integers such that $t_n = t_{\cC^{(n)}} \pg 1 + \frac{a+o(1)}{\ln n} \pd $. 
As $S\ni n\to \infty$, we have 
    \begin{equation}
\dtv(\Unif_{\cC^{(n)}}^{*t_n}, \Unif_{\kE({\cC^{(n)}},t_n)})  \to \dtv\pg \Poiss(1), \Poiss\pg 1+e^{-a}\pd \pd.
    \end{equation}
\end{theorem}

Given an even integer $n\geq 2$, the random involution walk on $\kS_n$ with parameter $p\in (0,1)$ is the walk driven by the measure
\begin{equation}\label{eq: RI def}
    \RI_{n,p} = \sum_{s=0}^{n/2} \binom{n/2}{s}(1-p)^{s}p^{n/2-s} \Inv_{n,s},
\end{equation}
where $\Inv_{n,s}$ is the uniform measure on permutations of $\kS_n$ with cycle type $[2^{s}, 1^{n-2s}]$, i.e.\ which are a product of $s$ disjoint transpositions.
Bernstein \cite{Bernstein2018involutions} found the order of magnitude of the mixing time for these walks with a lower bound at $\log_{1/p}n$, and conjectured a cutoff at this time. We prove this conjecture and furthermore find the cutoff profile.

\begin{theorem}\label{thm: profil pour les involutions aléatoires}
    Let $p\in (0,1)$.
    Let $a\in \bbR$. Let $S\subset \bbN$ and $(t_n)_{n\in S}$ be a sequence of integers such that $t_n =\frac{\ln n +a+o(1)}{\ln(1/p)} $.
    As $S \ni n\to \infty$ we have 
\begin{equation}
    \dtv \pg \pg\RI_{n,p}\pd^{*t_n}, \Unif_{\kS_n}\pd \to \dtv\pg  \Poiss\pg 1+e^{-a}\pd, \Poiss(1)\pd.
\end{equation}
\end{theorem}

Finally, we give estimates on the random transposition walk for $n=52$.
They rely on computers only for elementary computations. This illustrates that using representation theory can give very precise results not only asymptotically, but also for small values of $n$, and gives concrete estimates on the question at the origin of the study of the cutoff phenomenon. 

\begin{proposition}\label{prop: numerics for transpositions for n=52}
    Let $n=52$. Recall that $\cT$ denotes the conjugacy class of transpositions. Write $\tmix(\varepsilon) = \min \ag t\geq 0 \mid \dtv(\Unif_{\cT}^{*t}, \Unif_{\kE({\cT},t)}) \leq \varepsilon \ad$ for $\varepsilon\in (0,1)$. We have
\begin{enumerate}
    \item $187 \leq \tmix(10^{-2})\leq 191$;
    \item $246 \leq \tmix(10^{-3})\leq 247$;
    \item $\tmix(10^{-4}) = 304$.
\end{enumerate}
\end{proposition}
\begin{remark}
    Laziness has a significant impact for small values of $n$ for the original lazy walk (the $\pg \frac{1}{n} + \frac{n-1}{n}\Unif_{\cT}\pd$-walk (started at the identity permutation $\Id$), which we denote by $(X_t)_{t\geq 0}$). 
    
For $n=52$, at time 100, due to the small laziness $1/52$, the sign of permutations is not fully mixed yet, and one can even bet on $X_{100}$ being even.
More precisely, the number of times $\Id$ has been picked before time 100 follows a binomial law $\Bin(100, 1/52)$, which is approximately $\Poiss(\alpha)$, where $\alpha = 100/52$ (this approximation is precise). Therefore, the probability that $X_{100}$ is an even permutation is about the probability that a $\Poiss(\alpha)$ random variable is even, that is $\frac{\cosh(\alpha)}{e^\alpha} = \frac{1+e^{-2\alpha}}{2} \approx 51\%$, which is significant. Similarly, the probability that $X_{160}$ is even is about $50.1\%$.
\end{remark}

\subsection{Obtaining cutoff profiles with eigen-analysis}

There are different techniques to obtain cutoff profiles, using spectral analysis, probabilistic arguments, or a bit of both. For walks that have lots of symmetries, it is often possible to obtain detailed spectral information. Then to obtain the cutoff profile it is sufficient to have the following three ingredients:
\begin{enumerate}
    \item[(1)] precise-enough bounds on eigenvalues and their multiplicities;
    \item[(2)] precise-enough Taylor expansions of the eigenvalues that \textit{asymptotically contribute};
    \item[(3)] identities that relate eigenvectors corresponding to eigenvalues that \textit{asymptotically contribute}.
\end{enumerate}

In this paper, we focus on the $\Unif_{\cC}$-walks, where $\cC$ is a conjugacy class of $\kS_n$ with $\Theta(n)$ fixed points (or a mixture of such classes for random involutions). For these walks, (1) was obtained in \cite{Olesker-TaylorTeyssierThevenin2025sharpboundsandcutoff} and (3) was obtained in \cite{Teyssier2020}. The last piece needed to prove Theorem~\ref{thm: profil TV pour les CC} is therefore to prove asymptotic equivalents for some of the eigenvalues, i.e.\ to prove (2). 

\medskip

Let us give more context on methods leading to (1), (2), and (3), and how (2) differs from (1) and (3). This paragraph assumes some knowledge of representation theory.

For (1), one needs to obtain precise enough uniform bounds on characters (i.e.\ on eigenvalues). There is no known way to do this using only the Murnaghan--Nakayama rule. In \cite{Olesker-TaylorTeyssierThevenin2025sharpboundsandcutoff} we relied on the Naruse hook length formula (see \cite{MoralesPakPanova2018I}), and had to introduce deep triple decompositions of Young diagrams to quantify which parts of the diagrams contribute to characters and which parts contribute to dimensions (i.e.\ to multiplicities of eigenvalues).
This combinatorial technique was designed for \textit{high frequencies}, which are the most delicate to control despite not contributing asymptotically. The counting arguments worked well for conjugacy classes with, say, $\varepsilon n$ to $(1-\varepsilon)n$ fixed points, and some probabilistic bootstrapping was needed to transfer the bounds to conjugacy classes with support $o(n)$.

For (3), we explained in \cite[Section~2]{Teyssier2020} how to filter out \textit{low frequencies} (the ones which asymptotically contribute) and obtained, mostly via the Murnaghan--Nakayama rule, a character identity \cite[Lemma~4.3]{Teyssier2020} which tells how some weighted sums of characters (i.e.\ eigenvectors) at the same \textit{level} asymptotically compensate (for most permutations). We thank Christian Krattenthaler for identifying the polynomials from this identity to be Poisson--Charlier polynomials.

For (2), which is the goal of the present paper, we will prove asymptotic equivalents of characters for low frequencies in Theorem~\ref{thm: equivalent asymptotique r fini pour les caractères suffisamment de points fixes}. This also relies on the Murnaghan--Nakayama rule but in a different way: we study characters individually and we count how many ribbon tableaux of a Young diagram there are (the ribbons are also peeled in a different order compared to what was done for (3)), depending on the number of boxes over the first row which are covered by ribbons of length at least 2. We also complement these equivalents with bounds, such as Theorem~\ref{thm: borne asymptotique r fini pour les caractères peu de points fixes}, which can be proved by the same technique but are useful for other problems. We discuss some of these other potential applications at the end of Section~\ref{s: statements for applications}.

\subsection{Notation}

We denote the symmetric group of degree $n$ by $\kS_n$ and its set of irreducible representations by $\widehat{\mathfrak{S}_n}$. Since $\widehat{\mathfrak{S}_n}$ is in bijection with the set $\kY_n$ of integer partitions of size $n$, which can be represented by Young diagrams, for convenience we use $\lambda$ to denote at the same time a representation, the associated integer partition, and the associated Young diagram, and we may write $\lambda \vdash n$ for $\lambda \in \widehat{\kS_n}$.
If $\lambda\in \widehat{\mathfrak{S}_n}$ and $\sigma \in \mathfrak{S}_n$ is a permutation, we denote the dimension of $\lambda$ by $d_\lambda$, the associated character by $\ch^\lambda(\sigma) =\Tr \rho^\lambda(\sigma)$, and the associated renormalized character by $\chi^\lambda(\sigma) = \frac{\ch^\lambda(\sigma)}{d_\lambda}$.
We assume familiarity with the representation theory of symmetric groups and refer to \cite{LivreMeliot2017RepresentationTheoryofSymmetricGroups} for a detailed exposition.

\medskip

Given a Young diagram $\lambda \vdash n$, we denote the length of its $i$-th row (for $i\geq 1$) by $\lambda_i$.

\medskip

We write 
\begin{itemize}
    \item $f(n) = O(g(n))$ or $f(n) \lesssim g(n)$ or $g(n) = \Omega(f(n))$ if there exists a constant $C>0$ such that $|f(n)| \leq C |g(n)|$ for all $n$ large enough;
    \item $f(n) = o(g(n))$ or $g(n) = \omega(f(n))$ if $f(n)/g(n) \to 0$ as $n\to \infty$;
    \item $f(n) = \Theta(g(n))$ or $f(n) \asymp g(n)$ if $f(n) = O(g(n))$ and $g(n) = O(f(n))$;
    \item $f(n) \sim g(n)$ if $g(n) = (1+o(1)) f(n)$.
\end{itemize}
We often emphasize the dependence on a parameter by writing it as an index. For example, for $r>0$, as $n\to \infty$, we may write $5n + r^r\sqrt{n} \lesssim_r n$ or $(r^2+ r + 1) n! = O_r(n!)$.

\section{Asymptotics of finite-level characters}\label{s: Asymptotics of low-level characters}

\subsection{General statements}
We introduce a new quantity to better capture the contribution of the cycle structure of permutations to characters: for $\sigma \in \kS_n$ we set
\begin{equation}
    M(\sigma)   = \max_{i\geq 1} f_i(\sigma)^{1/i};
\end{equation}
and the rescaled version
\begin{equation}
    J(\sigma) = \frac{\ln M(\sigma)}{\ln n} = \max_{i\geq 1} \frac{1}{i}\frac{\ln f_i(\sigma)}{\ln n}.
\end{equation}
Recall that $\ch^\lambda$ denotes the character and $\chi^\lambda = \ch^\lambda/d_\lambda$ is the renormalized character. The \textit{level}, which we denote by $r(\lambda)$, of a representation (or equivalently of a Young diagram) $\lambda\vdash n$ is the number of boxes outside of its first row, that is, $r(\lambda) = n-\lambda_1$. Intuitively, representations with a low level correspond to low frequencies in terms of Fourier analysis.

We prove the following general bound on finite-level characters of symmetric groups.
\begin{theorem}\label{thm: character bound finite level general with M sigma}
    Let $r\geq 0$. As $n\to \infty$, uniformly over all $\lambda \vdash n$ such that $\lambda_1 = n-r$ and every $\sigma \in \kS_n$, we have
    \begin{equation}
        \bg \ch^\lambda(\sigma) \bd \lesssim_r M(\sigma)^r;
    \end{equation}
and equivalently,
\begin{equation}\label{eq: form with J sigma}
     \bg \ch^\lambda(\sigma) \bd \lesssim_r d_\lambda^{J(\sigma)}.
\end{equation}
\end{theorem}
\begin{remark}
    Larsen and Shalev \cite{LarsenShalev2008} introduced an orbit growth exponent $E(\sigma)$ and proved a uniform bound $\bg \ch^\lambda(\sigma) \bd \leq d_\lambda^{E(\sigma) + o(1)}$, which we improved to $\bg \ch^\lambda(\sigma) \bd \leq d_\lambda^{E(\sigma) + O(1/\ln n)}$ in \cite{TeyssierThevenin2025virtualdegreeswitten}. For representations of level $r$, this bound can be rewritten as 
\begin{equation}
    \bg \ch^\lambda(\sigma) \bd \lesssim_r d_\lambda^{E(\sigma)}.
\end{equation}
The quantity $E(\sigma)$ is optimal for permutations $\sigma$ with a \textit{pure} cycle structure, that is if $\sigma \sim m^{n/m}$ for some $m$. In this case we have $J(\sigma)= E(\sigma)$. However, $E(\sigma)$ fails to give the correct exponent for hybrid cycle structures.
For example, omitting integer parts, if $\sigma$ is an involution with $\sqrt{n}$ fixed points, i.e.\ if $\sigma \sim (1^{\sqrt{n}}$, we have $2^{(n-\sqrt{n})/2})$, $E(\sigma) = 3/4 + o(1)$ while $J(\sigma) = 1/2$.
\end{remark}

We conjecture that the sharpest exponential character bound, uniform over all $\lambda \in \widehat{\kS_n}$ and all $\sigma\in \kS_n$, should be written in terms of $J(\sigma)$, with an error term that is so small that it does not affect the applications to mixing times. 

\begin{remark}
    The main result of \cite{Olesker-TaylorTeyssierThevenin2025sharpboundsandcutoff} is a character bound that is uniform over all $\lambda$ and all $\sigma$ with $\Theta(n)$ fixed points, and that can be rewritten as $\bg \ch^\lambda(\sigma) \bd \leq d_\lambda^{J(\sigma)}$. However the bound for permutations with $o(n)$ fixed points are currently weaker than this.
\end{remark}

\begin{theorem}\label{thm: character estimate many fixed points with error term}
    Let $r\geq 1$. As $n\to \infty$, uniformly over all diagrams $\lambda \vdash n$ with $n-\lambda_1 = r$, and all permutations $\sigma \in \kS_n$ with $f_1(\sigma) \geq \sqrt{n}$, we have, denoting $f = f_1(\sigma)$ and $k=n-f$,
\begin{equation}
        \chi^\lambda(\sigma) = \pg \frac{f}{n}\pd^r \pg 1 +O_r\pg \frac{k}{f^2} \pd \pd. 
    \end{equation}
\end{theorem}
\begin{remark}
The case of [$k$-cycles with $k\leq n/3$] of Theorem~\ref{thm: character estimate many fixed points with error term} is known as \cite[Corollary~3.4]{NestoridiOlesker-Taylor2022limitprofiles}, and the case of finite $k$-cycles was reproved independently in \cite[Lemma~4.3]{Fulman2024commutators}, both with the same the error term. Both results above rely on the character estimates of \cite{Hough2016}, which rely on complex analysis.
\end{remark}

\begin{remark}\label{rk: sign change if transpose and odd}
   Let $n\geq 2$, $\sigma \in \kS_n$ and $\lambda \vdash n$. Then $\chi^{\lambda'}(\sigma) = \sgn(\sigma) \chi^\lambda(\sigma)$, where $\lambda'$ is the conjugate of $\lambda$. Therefore Theorem~\ref{thm: character estimate many fixed points with error term} also applies up to a factor $\sgn(\sigma)$ to diagrams with $\lambda_1' = n-O(1)$.
\end{remark}

\subsection{Statements for applications}\label{s: statements for applications}

The following consequences of Theorems~\ref{thm: character bound finite level general with M sigma} and~\ref{thm: character estimate many fixed points with error term} will be useful in practice. Corollary~\ref{cor: equivalent asymptotique d lambda fois caractère à une certaine puissance r fini pour les caractères suffisamment de points fixes} is a generalization of \cite[Lemma~4.4]{Fulman2024commutators} (which is implicit in \cite{NestoridiOlesker-Taylor2022limitprofiles}). It follows from Theorem~\ref{thm: equivalent asymptotique r fini pour les caractères suffisamment de points fixes}, which itself follows from Theorem~\ref{thm: character estimate many fixed points with error term}. Theorem~\ref{thm: borne asymptotique r fini pour les caractères peu de points fixes} and Corollary~\ref{cor: equivalent asymptotique d lambda fois caractère à une certaine puissance r fini pour les caractères peu de points fixes} are sharp versions of \cite[Corollary~7.5~(a)]{TeyssierThevenin2025virtualdegreeswitten}, restricted to finite-level representations.

\begin{theorem}\label{thm: equivalent asymptotique r fini pour les caractères suffisamment de points fixes} 
Let $r\geq 0$. For each $n\geq 1$, let $\sigma_n\in \kS_n$.
Assume that $f_1(\sigma_n)/\sqrt{n}\to \infty$. Then as $n\to \infty$, uniformly over all $\lambda \vdash n$ such that $\lambda_1 = n-r$, we have
\begin{equation}
    \chi^{\lambda}(\sigma_n) = \pg f_1(\sigma_n)/n\pd^{r + o_r(1/\ln n)}.
\end{equation}
\end{theorem}
\begin{corollary}\label{cor: equivalent asymptotique d lambda fois caractère à une certaine puissance r fini pour les caractères suffisamment de points fixes} 
    Let $r\geq 0$. Let $c\in \bbR$. For each $n\geq 1$, let $\sigma_n\in \kS_n$. Let $(t_n)$ be a sequence of positive real numbers such that $t_n = \frac{(\ln n)+c +o(1)}{\ln(n/f_1(\sigma_n))}$.
Assume that $f_1(\sigma_n)/\sqrt{n}\to \infty$. Then as $n\to \infty$, uniformly over all $\lambda \vdash n$ such that $\lambda_1 = n-r$, we have
\begin{equation}
    d_\lambda \chi^{\lambda}(\sigma_n)^{t_n}\to \frac{e^{-rc}}{r!}d_{\lambda^*},
\end{equation}
where $\lambda^* = (\lambda_2, \lambda_3, \ldots) \vdash r$ is the diagram $\lambda$ minus its first row.
\end{corollary}

\begin{theorem}\label{thm: borne asymptotique r fini pour les caractères peu de points fixes} 
Let $r\geq 0$. For each $n\geq 1$, let $\sigma_n\in \kS_n$.
Assume that $f_1(\sigma_n) = o(\sqrt{n})$ and $f_2(\sigma_n) = o(n)$. Then as $n\to \infty$, uniformly over all $\lambda \vdash n$ such that $\lambda_1 = n-r$, we have
\begin{equation}
    \chi^{\lambda}(\sigma_n) = o_r(n^{-r/2}).
\end{equation}
\end{theorem}
\begin{corollary}\label{cor: equivalent asymptotique d lambda fois caractère à une certaine puissance r fini pour les caractères peu de points fixes} 
   Let $r\geq 0$. For each $n\geq 1$, let $\sigma_n\in \kS_n$.
Assume that $f_1(\sigma_n) = o(\sqrt{n})$ and $f_2(\sigma_n) = o(n)$. Then as $n\to \infty$, uniformly over all $\lambda \vdash n$ such that $\lambda_1 = n-r$, we have
\begin{equation}
    d_\lambda \chi^{\lambda}(\sigma_n)^{2}\to 0.
\end{equation}
\end{corollary}

\begin{remark}
    It is not hard to check that the assumptions “$f_1(\sigma_n)/\sqrt{n}\to \infty$” and “$f_1(\sigma_n)=o(\sqrt{n})$ and $f_2(\sigma_n) = o(n)$” in the results above are sharp, by looking at diagrams with two rows, i.e.\ diagrams $\lambda = (n-r,r)$.
\end{remark}

For applications to cutoff profiles, we will only need Theorem~\ref{thm: equivalent asymptotique r fini pour les caractères suffisamment de points fixes}. 

\medskip

Asymptotic equivalents of characters are useful to obtain cutoff profiles, but also on their own (even if one is not able to bound characters uniformly) to obtain information on the statistics of random permutations. Fulman \cite[Theorem~4.5]{Fulman2024commutators} proved this way that asymptotically the number of fixed points of a product of about $(1/k)n(\ln n + c)$ $k$-cycles is $\Poiss(1+e^{-c})$ (we note that similar ideas were soon after independently developed by Nestoridi and Yan \cite{NestoridiYan2025BiasedTranspositions} to compute the number of fixed points for biased transpositions), and Theorem~\ref{thm: points fixes pour les CC} generalizes this to conjugacy classes. Fulman's results were also generalized by Arcona \cite{Arcona2026smallcyclestructure}, who obtained limiting results for the structure of small cycles in a product of $k$-cycles of $\kS_n$. Corollary~\ref{cor: equivalent asymptotique d lambda fois caractère à une certaine puissance r fini pour les caractères suffisamment de points fixes} may be used to further generalize these results to conjugacy classes with $\omega(\sqrt{n})$ fixed points. On the other hand, the small cycle structure was proved to be close to that of a uniform permutation for some products of two conjugacy invariant random permutations in \cite{KammounMaida2020conjugacyinvariant} and in \cite{BudzinskiCurienPetri2019}.
We believe that combining Arcona's arguments with Theorem~\ref{thm: borne asymptotique r fini pour les caractères peu de points fixes} would enable recovering these results.

Since its proof is very short, we include here the generalization of \cite[Theorem~4.5]{Fulman2024commutators}, which shows that a product of about $\frac{\ln n + a + o(1)}{\ln(n/f_1(\cC^{(n)}))}$ independent $\Unif_{\cC^{(n)}}$-permutations has in law about $\Poiss(1+e^{-a})$ fixed points.

\begin{theorem}\label{thm: points fixes pour les CC}
For each $n$ let $ \cC^{(n)} \in\Conj^*(\kS_n)$, and assume that $f_1(\cC^{(n)})/\sqrt{n}\to \infty$ as $n\to \infty$. Let  $a\in \bbR$. Assume that there exist $S\subset \bbN^*$ and a sequence $(t_n)_{n\in S}$ of integers such that $t_n = t_{\cC^{(n)}} \pg 1 + \frac{a+o(1)}{\ln n} \pd$, and let, for $n\in S$, $Y_n$ be a $\Unif_{\cC^{(n)}}^{*t_n}$ random variable. Then, as $S\ni n\to \infty$, $f_1(Y_n)$ converges in distribution to $\Poiss\pg 1+e^{-a}\pd$.
\end{theorem}
\begin{proof}
The proof of \cite[Theorem~4.5]{Fulman2024commutators} carries verbatim, except for the use of \cite[Lemma~4.4]{Fulman2024commutators} to obtain Equation 12, which is replaced by Corollary~\ref{cor: equivalent asymptotique d lambda fois caractère à une certaine puissance r fini pour les caractères suffisamment de points fixes}.
\end{proof}

\subsection{The Murnaghan--Nakayama rule}
Our proofs rely on the Murnaghan--Nakayama rule (see for instance \cite[Theorem~3.10]{LivreMeliot2017RepresentationTheoryofSymmetricGroups}) which is an explicit combinatorial formula that enables computing characters.

We refer the reader unfamiliar with the Murnaghan--Nakayama rule (and the related notions of skew partition, ribbon, ribbon tableau, height) to a companion paper \cite[Section~3.1]{Teyssier2025SkewDimensionsAndCharacters}, which explains it with a fair amount of details and illustrations.

We use the following notation:
\begin{itemize}
    \item $d_{\lambda\backslash \mu}$ for the number of standard tableaux of a skew partition $\lambda \backslash \mu$;
    \item $\RT(\lambda, \alpha)$ for the set of ribbon tableaux of shape $\lambda$ and weight $\alpha$;
    \item $\RT(\lambda)$ for the set of all ribbon tableaux of shape $\lambda$ (and of any weight);
    \item $\height(\rho)$ for the height of a ribbon $\rho$;
    \item $\height(T) = \sum_{\rho\in T} \height(\rho)$ for the height of a ribbon tableau $T$.
\end{itemize}
The Murnaghan--Nakayama rule can then be stated as follows.
If $n\geq 1$, $\lambda\vdash n$, and $\sigma\in \mathfrak{S}_n$, writing $\alpha$ for the cycle lengths of $\sigma$ (in any order), we have
\begin{equation}
    \ch^\lambda(\sigma) = \sum_{T \in \RT(\lambda, \alpha)} (-1)^{\height(T)}.
\end{equation}

\subsection{Estimates on dimensions of representations}
There are many bounds on dimensions of irreducible representations in the literature. We will use the following one, which follows directly from the hook length formula. Recall that given a Young diagram $\lambda \vdash n$, we write $r(\lambda) = n- \lambda_1$.
\begin{lemma}[{\cite[Proposition~3.2]{Teyssier2020}}]\label{lem: asymptotic hook length formula}
    Let $r\geq 0$. As $n\to \infty$, uniformly over all $\lambda \vdash n$ such that $r(\lambda) = r$, we have
    \begin{equation}
            d_\lambda = \binom{n}{r} d_{\lambda^*} \pg 1 -\frac{r}{n} + O_r\pg \frac{1}{n^2}\pd\pd.
        \end{equation}
\end{lemma}
We also prove a second order expansion on descending factorials.
Given two integers $0\leq r \leq n$ we write $n^{\downarrow r} = r! \binom{n}{r} = n(n-1)\cdots(n-r+1)$. For $m\geq 0$, denote the $m$-th triangular number by $T(m) = \sum_{i=1}^m i = \frac{m(m+1)}{2}$.
\begin{lemma}
    Let $r\geq 1$. As $n\to \infty$ we have
    \begin{equation}
        n^{\downarrow r} = n^r \pg 1 - \frac{T(r-1)}{n} + O_r\pg \frac{1}{n^2} \pd\pd.
    \end{equation}
\end{lemma}
\begin{proof}
    Expanding the product (which has finitely many terms), and using that $r = O_r(1)$, we get
\begin{equation}\label{lem: estimate on descending factorials}
    \frac{n^{\downarrow r}}{n^r} = \prod_{i=0}^{r-1}\pg 1-\frac{i}{n}\pd = 1 - \frac{\sum_{i=0}^{r-1}i}{n} + O_r\pg \frac{1}{n^2} \pd = 1 - \frac{T(r-1)}{n} + O_r\pg \frac{1}{n^2} \pd. \qedhere
\end{equation}
\end{proof}
We deduce the following.
\begin{lemma}\label{lem: asymp of dimension ratio with r K f}
    Let $r,K,f$ be three integers such that $0\leq K \leq r\leq f/2$. As $n\to \infty$, uniformly over all $\lambda \vdash n$ such that $r(\lambda)= r$ and all $\mu \vdash f$ such that $\mu \subset \lambda$ and $r(\mu) = r - K$, we have
    \begin{equation}
     \frac{d_{\mu}}{d_\lambda} \asymp_r \pg\frac{f}{n} \pd^r f^{-K};
    \end{equation}
and moreover, if $K=0$ (that is, if $\mu^* = \lambda^*$), we have, denoting $k = n-f$, 
\begin{equation}
    \frac{d_{\mu}}{d_\lambda} = \pg \frac{f}{n} \pd^r \pg 1 - T(r)\frac{k}{fn} + O_r\pg \frac{1}{f^2}\pd \pd.
\end{equation}
\end{lemma}
\begin{proof}
    Uniformly over all $\lambda$ and $\mu$ as above, by Lemma~\ref{lem: asymptotic hook length formula} we have $d_\lambda \asymp_r \binom{n}{r}\asymp_r  n^r$ and $d_\mu \asymp_r \binom{f}{r-K} \asymp_r f^{r-K}$, so $\frac{d_{\mu}}{d_\lambda} \asymp_r \pg f/n \pd^r f^{-K}$, which proves the first point. For the second point, assuming that $K=0$, we have $d_{\lambda^*} = d_{\mu^*}$, so by Lemma~\ref{lem: asymptotic hook length formula} we have
    \begin{equation}
        \frac{d_{\mu}}{d_\lambda} = \frac{\binom{f}{r}}{\binom{n}{r}}\pg 1 - r \pg \frac{1}{f}- \frac{1}{n}\pd + O_r\pg \frac{1}{f^2} \pd \pd.
    \end{equation}
Moreover, by Lemma~\ref{lem: estimate on descending factorials}, we have
\begin{equation}
   \frac{\binom{f}{r}}{\binom{n}{r}} =  \pg\frac{f}{n} \pd^{r} \pg 1 - T(r-1)\pg \frac{1}{f} - \frac{1}{n} \pd + O_r\pg \frac{1}{f^2} \pd \pd.
\end{equation}
The result follows, since $T(r-1) + r = T(r)$ and $\frac{1}{f} - \frac{1}{n} = \frac{k}{fn}$.
\end{proof}
\subsection{Upper configurations}

For the proofs it will be convenient to introduce a truncated version of $M(\sigma)$: for $\sigma \in \kS_n$, we set
\begin{equation}
     G(\sigma) = \max_{i\geq 2} f_i(\sigma)^{1/i}.
\end{equation}

\begin{lemma}\label{lem: technical lemma bound on S}
    Let $r\geq 1$.
    Let $n\geq 3r$. Let $\lambda \vdash n$ such that $r(\lambda) = r$. Let $\sigma \in \kS_n$ such that $f:= f_1(\sigma) \geq 2r$, and denote its cycle structure written in weakly increasing order by $\alpha$. Let $\mu \vdash f$ be a sub Young diagram of $\lambda$. Denote $K = |\lambda^*\backslash \mu^*|$. Assume that $K\geq 2$, and let $s=(s_i)_{i\geq 2}$ such that $\sum_{i\geq 2} is_i = K$. Denote the set of ribbon tableaux of $\lambda$ of weight $\alpha$, which cover $\lambda^*\backslash \mu^*$ with $s_i$ ribbons of length $i$ for each $i\geq 2$, by $S(\lambda, \alpha, \mu, s)$. Then
    \begin{equation}
        |S(\lambda, \alpha, \mu, s)| \lesssim_r
        f^{r-K} G(\sigma)^K \leq M(\sigma)^r.
    \end{equation}
\end{lemma}
\begin{proof}
First, since $f \geq 2r$ (and in particular $f \geq r + \lambda_2$) we have $\mu_1 \geq \lambda_2$, so the first row of the skew diagram $\nu := \lambda \backslash \mu$ is disconnected from the rest, as illustrated in Figure~\ref{fig: skew diagram 1}. 
\begin{figure}[!ht]
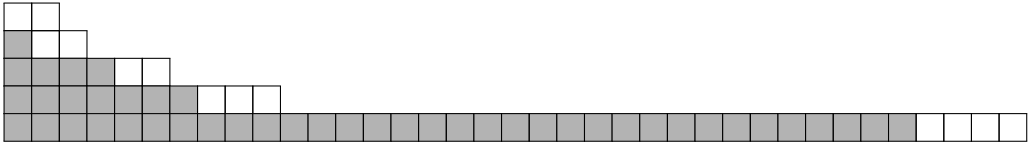

    \centering
{
\ytableausetup{boxsize=1em}
\begin{ytableau}
    \none && \\
    \none &*(gray!60)&&\\
    \none &*(gray!60)&*(gray!60)&*(gray!60)&*(gray!60)&&\\
    \none &*(gray!60)&*(gray!60)&*(gray!60)&*(gray!60)&*(gray!60)&*(gray!60)&*(gray!60)&&&\\
    \none &*(gray!60)&*(gray!60)&*(gray!60)&*(gray!60)&*(gray!60)&*(gray!60)&*(gray!60)&*(gray!60)&*(gray!60)&*(gray!60)&*(gray!60)&*(gray!60)&*(gray!60)&*(gray!60)&*(gray!60)&*(gray!60)&*(gray!60)&*(gray!60)&*(gray!60)&*(gray!60)&*(gray!60)&*(gray!60)&*(gray!60)&*(gray!60)&*(gray!60)&*(gray!60)&*(gray!60)&*(gray!60)&*(gray!60)&*(gray!60)&*(gray!60)&*(gray!60)&*(gray!60)&&&&  \\
    \none
\end{ytableau}
}

    \caption{The diagram $\lambda = (37,10,6,3,2)\vdash 58$ consists of all cells, the sub-diagram $\mu = (33, 7, 4, 1)\vdash 45$ consists of all gray cells, and the skew diagram $\nu = \lambda\backslash \mu$ consists of all white cells.}
    \label{fig: skew diagram 1}
\end{figure}

Let us now (upper-)count how many diagrams there are in $S(\lambda, \alpha, \mu, s)$.
First, $\mu$ is covered by 1-ribbons in a standard way; there are $d_\mu$ such configurations (independently of the rest of the ribbon configuration).
Also, given  $s_i$ ribbons of length $i$ for each $i\geq 2$, there are $O_r(1)$ ways to cover $\lambda^* \backslash \mu^*$ with these ribbons, as illustrated in Figure~\ref{fig: upper skew covered by ribbons}.
\begin{figure}[!ht]
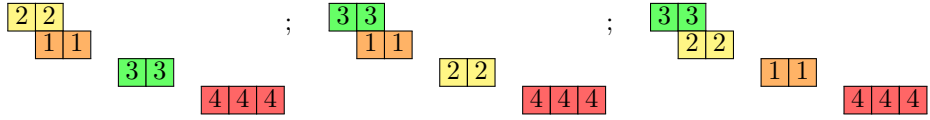

    \centering
{
\ytableausetup{boxsize=1em}
\begin{ytableau}
    \none &*(yellow!60)2&*(yellow!60)2\\
    \none &\none&*(orange!60)1&*(orange!60)1\\
    \none &\none&\none&\none&\none&*(green!60)3&*(green!60)3\\
    \none &\none&\none&\none&\none&\none&\none&\none&*(red!60)4&*(red!60)4&*(red!60)4
\end{ytableau};
\begin{ytableau}
    \none &*(green!60)3&*(green!60)3 \\
    \none &\none&*(orange!60)1&*(orange!60)1\\
    \none &\none&\none&\none&\none&*(yellow!60)2&*(yellow!60)2\\
    \none &\none&\none&\none&\none&\none&\none&\none&*(red!60)4&*(red!60)4&*(red!60)4
\end{ytableau};
\begin{ytableau}
    \none &*(green!60)3&*(green!60)3 \\
    \none &\none&*(yellow!60)2&*(yellow!60)2\\
    \none &\none&\none&\none&\none&*(orange!60)1&*(orange!60)1\\
    \none &\none&\none&\none&\none&\none&\none&\none&*(red!60)4&*(red!60)4&*(red!60)4
\end{ytableau}
}
\caption{All possible ribbon coverings of the skew diagram $ (10,6,3,2)\backslash (7,4,1)$, with weight $(2,2,2,3)$.
Colors are redundant.}
    \label{fig: upper skew covered by ribbons}
\end{figure}
Since there are $\prod_{i\geq 2} \binom{f_i(\sigma)}{s_i}$ ways to pick  $s_i$ ribbons of length $i$ for each $i\geq 2$, this proves that 
 \begin{equation}
        |S(\lambda, \alpha, \mu, s)| \lesssim_r d_\mu \prod_{i\geq 2} \binom{f_i(\sigma)}{s_i}.
    \end{equation}
Now, by the definition of $G(\sigma)$, we have
\begin{equation}
    \prod_{i\geq 2} \binom{f_i(\sigma)}{s_i} \leq \prod_{i\geq 2} f_i(\sigma)^{s_i} = \prod_{i\geq 2} (f_i(\sigma)^{1/i})^{is_i} \leq \prod_{i\geq 2} G(\sigma)^{is_i} = G(\sigma)^K,
\end{equation}
and $d_\mu \lesssim f^{r-K}$ (for example by Lemma~\ref{lem: asymptotic hook length formula}), which concludes the proof of the first point. The second point follows immediately, since $\max(f, G(\sigma))= M(\sigma)$ by definition.
\end{proof}

\begin{lemma}\label{lem: technical lemma bound on S prime} 
    Let $r\geq 1$.
    Let $n\geq 3r$. Let $\lambda \vdash n$ such that $r(\lambda) = r$. Let $\sigma \in \kS_n$ such that $f:= f_1(\sigma) \geq 2r$, and denote its cycle structure written in weakly increasing order by $\alpha$. Let $K\in \ag 0, 2,3,4, \ldots\ad$. Let $S'(\lambda, \alpha, K)$ be the set of all ribbon tableaux of $\lambda$ with weight $\alpha$ such that all but $K$ boxes in $\lambda^*$ are filled with 1-ribbons. Then
\begin{equation}
    |S'(\lambda, \alpha, K)| \lesssim_r
    f^{r-K} G(\sigma)^K \leq M(\sigma)^r.
\end{equation}
Moreover, $S'(\lambda, \alpha, 0)$ consists of $d_{(\lambda_1-k, \lambda_2, \lambda_3, \ldots)}$ ribbon tableaux of height 0.
\end{lemma}
\begin{proof}
 There are $O_r(1)$ many diagrams $\mu \vdash f$  such that $\mu \subset \lambda$ and $r(\mu) = r(\lambda) - K$. For each such $\mu$ there are $O_K(1) = O_r(1)$ choices of tuples of non-negative integers $s = (s_2, s_3, \ldots)$ such that $\sum_{i\geq 2}is_i = K$. Finally, given such an $s$, there are $O_r(1)$ ribbon tableaux of $\lambda^*\backslash \mu^*$ having $s_i$ ribbons of length $i$ for each $i\geq 2$. 
We therefore have 
\begin{equation}
    |S'(\lambda, \alpha, K)| \lesssim_r \max_{\mu, s} |S(\lambda, \alpha, \mu, s)|,
\end{equation}
and hence the first claim follows from Lemma~\ref{lem: technical lemma bound on S}.

If $K = 0$, then the diagram $\mu = (\lambda_1-k, \lambda_2, \lambda_3, \ldots)$ is covered by 1-ribbons -- there are $d_\mu$ such possibilities since $d_\mu$ is the number of standard tableaux of $\mu$ -- and the ribbons of length at least 2 must be ordered at the end of the first row, as illustrated in Figure~\ref{fig: ribbon tableaux for K=0}, so there is only one possibility to place the ribbons of length at least 2.
\begin{figure}[!ht]
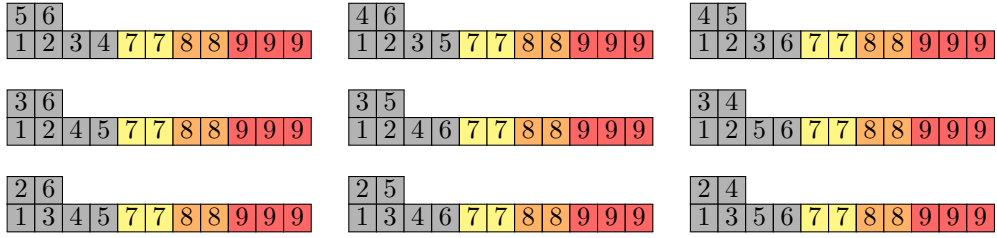

    \centering
{
\ytableausetup{boxsize=1em}
\begin{ytableau}
    \none &*(gray!60)5&*(gray!60)6 \\
    \none &*(gray!60)1&*(gray!60)2&*(gray!60)3&*(gray!60)4 &*(yellow!60)7&*(yellow!60)7&*(orange!60)8&*(orange!60)8&*(red!60)9&*(red!60)9&*(red!60)9\\
    \none 
\end{ytableau}
\begin{ytableau}
    \none &*(gray!60)4&*(gray!60)6 \\
    \none &*(gray!60)1&*(gray!60)2&*(gray!60)3&*(gray!60)5&*(yellow!60)7&*(yellow!60)7&*(orange!60)8&*(orange!60)8&*(red!60)9&*(red!60)9&*(red!60)9 \\
    \none 
\end{ytableau}
\begin{ytableau}
    \none &*(gray!60)4&*(gray!60)5 \\
    \none &*(gray!60)1&*(gray!60)2&*(gray!60)3&*(gray!60)6&*(yellow!60)7&*(yellow!60)7&*(orange!60)8&*(orange!60)8&*(red!60)9&*(red!60)9&*(red!60)9 \\
    \none 
\end{ytableau}

\begin{ytableau}
    \none &*(gray!60)3&*(gray!60)6 \\
    \none &*(gray!60)1&*(gray!60)2&*(gray!60)4&*(gray!60)5&*(yellow!60)7&*(yellow!60)7&*(orange!60)8&*(orange!60)8&*(red!60)9&*(red!60)9&*(red!60)9 \\
    \none 
\end{ytableau}
\begin{ytableau}
    \none &*(gray!60)3&*(gray!60)5 \\
    \none &*(gray!60)1&*(gray!60)2&*(gray!60)4&*(gray!60)6&*(yellow!60)7&*(yellow!60)7&*(orange!60)8&*(orange!60)8&*(red!60)9&*(red!60)9&*(red!60)9 \\
    \none 
\end{ytableau}
\begin{ytableau}
    \none &*(gray!60)3&*(gray!60)4 \\
    \none &*(gray!60)1&*(gray!60)2&*(gray!60)5&*(gray!60)6&*(yellow!60)7&*(yellow!60)7&*(orange!60)8&*(orange!60)8&*(red!60)9&*(red!60)9&*(red!60)9 \\
    \none 
\end{ytableau}

\begin{ytableau}
    \none &*(gray!60)2&*(gray!60)6 \\
    \none &*(gray!60)1&*(gray!60)3&*(gray!60)4&*(gray!60)5&*(yellow!60)7&*(yellow!60)7&*(orange!60)8&*(orange!60)8&*(red!60)9&*(red!60)9&*(red!60)9
\end{ytableau}
\begin{ytableau}
    \none &*(gray!60)2&*(gray!60)5 \\
    \none &*(gray!60)1&*(gray!60)3&*(gray!60)4&*(gray!60)6&*(yellow!60)7&*(yellow!60)7&*(orange!60)8&*(orange!60)8&*(red!60)9&*(red!60)9&*(red!60)9 
\end{ytableau}
\begin{ytableau}
    \none &*(gray!60)2&*(gray!60)4 \\
    \none &*(gray!60)1&*(gray!60)3&*(gray!60)5&*(gray!60)6&*(yellow!60)7&*(yellow!60)7&*(orange!60)8&*(orange!60)8&*(red!60)9&*(red!60)9&*(red!60)9 
\end{ytableau}

\begin{ytableau}
    \none 
\end{ytableau}
}
\caption{All possible ribbon coverings of the diagram $ (11,2)$, for $\alpha = (1,1,1,1,1,1,2,2,3)$. The ribbons of size 1 are filled with the numbers from 1 to 6, the first ribbon of size 2 is filled with 7's, the second is filled with 8's, and the ribbon of size 3 is filled with 9's. Colors are redundant.}
    \label{fig: ribbon tableaux for K=0}
\end{figure}
Also, all these ribbons are flat so the height of such a ribbon tableau is 0, and the second claim follows.
\end{proof}

\subsection{Proofs of Theorem~\ref{thm: character estimate many fixed points with error term}, Theorem~\ref{thm: equivalent asymptotique r fini pour les caractères suffisamment de points fixes}, and Corollary~\ref{cor: equivalent asymptotique d lambda fois caractère à une certaine puissance r fini pour les caractères suffisamment de points fixes}}
We can already prove a first block of character estimates.

\begin{proof}[Proof of Theorem~\ref{thm: character estimate many fixed points with error term}]
Denote the cycle structure of $\sigma$ written in weakly increasing order by $\alpha$. First, the assumption $f\geq \sqrt{n}$ ensures that $f\geq G(\sigma)$, i.e.\ that $M(\sigma) = f$. Also, we have 
\begin{equation}
    G(\sigma) = \max_{i\geq 2} f_i(\sigma)^{1/i} \leq \max_{i\geq 2} f_i(\sigma)^{1/2} \leq k^{1/2}.
\end{equation}
Therefore by Lemma~\ref{lem: technical lemma bound on S prime} we have
\begin{equation}
    \sum_{K=2}^r |S'(\lambda, \alpha, K)|  \lesssim_r \sum_{K=2}^r  f^{r-K} G(\sigma)^K =  \sum_{K=2}^r f^r \pg \frac{k^{1/2}}{f} \pd^K \lesssim_r f^r \pg \frac{k^{1/2}}{f} \pd^2 = f^r \frac{k}{f^2}.
\end{equation}
Moreover, since $f\geq 2r$, $\RT(\lambda, \alpha)$ is the disjoint union of the $S'(\lambda, \alpha, K)$ for $K \in \ag 0, 2,3, \ldots, r\ad$, and by Lemma~\ref{lem: technical lemma bound on S prime} again, the $d_{(\lambda_1-k, \lambda_2, \lambda_3, \ldots)}$ elements of $S'(\lambda, \alpha, 0)$ have height 0. Therefore
\begin{equation}
    \sum_{T \in S'(\lambda, \alpha, 0)}(-1)^{\height(T)} = d_{(\lambda_1-k, \lambda_2, \lambda_3, \ldots)},
\end{equation}
and by the triangle inequality and recalling that $d_\mu \asymp_r f^r$, we obtain
\begin{equation}
    \bg \sum_{K=2}^r \sum_{T \in S'(\lambda, \alpha, K)}(-1)^{\height(T)} \bd \leq \sum_{K=2}^r |S'(\lambda, \alpha, K)| \lesssim_r f^r  \frac{k}{f^2} \lesssim_r d_{(\lambda_1-k, \lambda_2, \lambda_3, \ldots)} \frac{k}{f^2}.
\end{equation}
We deduce from the Murnaghan--Nakayama rule that
\begin{equation}
\begin{split}
     \ch^\lambda(\sigma) 
     & = \sum_{T \in S'(\lambda, \alpha, 0)}(-1)^{\height(T)} + \sum_{K=2}^r \sum_{T \in S'(\lambda, \alpha, K)}(-1)^{\height(T)} \\
     & = d_{(\lambda_1-k, \lambda_2, \lambda_3, \ldots)} + d_{(\lambda_1-k, \lambda_2, \lambda_3, \ldots)}O_r\pg \frac{k}{f^2}\pd.
\end{split}
\end{equation}
Dividing both sides by $d_\lambda$ gives 
\begin{equation}
       \chi^\lambda(\sigma) = \frac{d_{(\lambda_1-k, \lambda_2, \lambda_3, \ldots)}}{d_\lambda}\pg 1 +O_r\pg \frac{k}{f^{2}} \pd \pd.
 \end{equation}
 The result follows, since  $\frac{d_{(\lambda_1-k, \lambda_2, \lambda_3, \ldots)}}{d_\lambda} = \pg \frac{f}{n} \pd^r \pg 1 + O_r\pg \frac{k}{f^2}\pd \pd$ by the second part of Lemma~\ref{lem: asymp of dimension ratio with r K f}.
\end{proof}
\begin{proof}[Proof of Theorem~\ref{thm: equivalent asymptotique r fini pour les caractères suffisamment de points fixes}]
    Denote $f = f_1(\sigma_n)$ and $k = n- f$. By Theorem~\ref{thm: character estimate many fixed points with error term}, we have
\begin{equation}
    \chi^\lambda(\sigma_n) = (f/n)^r(1+O_r(k/f^2)) = (f/n)^r e^{O_r(k/f^2)} = (f/n)^{r + O_r\pg \frac{k}{f^2 \ln(n/f)} \pd}.
\end{equation}
If $f\leq n^{2/3}$ then $\frac{k}{f^2 \ln(n/f)} \asymp \frac{n}{f^2 \ln n} = o(1/\ln n)$ since $f/\sqrt{n}\to \infty$ by assumption. If $f > n^{2/3}$ then, using that $\ln(n/f) = - \ln (1-k/n) \geq k/n$, we have $\frac{k}{f^2 \ln(n/f)} \leq \frac{n}{f^2} \leq \frac{n}{n^{4/3}} = o(1/\ln n)$.
This concludes the proof.
\end{proof}
\begin{proof}[Proof of Corollary~\ref{cor: equivalent asymptotique d lambda fois caractère à une certaine puissance r fini pour les caractères suffisamment de points fixes}]
    By Lemma~\ref{lem: asymptotic hook length formula} we have $d_\lambda \sim \frac{n^r}{r!}d_{\lambda^*}$. By Theorem~\ref{thm: equivalent asymptotique r fini pour les caractères suffisamment de points fixes}, we have
    \begin{equation}
\chi^{\lambda}(\sigma)^{\frac{(\ln n)+c+o(1)}{\ln(n/f_1(\sigma))}} = e^{-r((\ln n) + c + o(1))} \sim n^{-r} e^{-rc}.
    \end{equation}
Multiplying the two estimates gives the desired result.
\end{proof}

\subsection{Proofs of Theorem~\ref{thm: character bound finite level general with M sigma}, Theorem~\ref{thm: borne asymptotique r fini pour les caractères peu de points fixes}, and Corollary~\ref{cor: equivalent asymptotique d lambda fois caractère à une certaine puissance r fini pour les caractères peu de points fixes}}

The arguments above already enable proving Theorem~\ref{thm: character bound finite level general with M sigma} under the additional assumption “$f_1(\sigma)\geq 2r$”, but we keep it as an intermediate bound in terms of the number of ribbon tableaux for the moment.

\begin{lemma}\label{lem: character bound finite level general with M sigma with extra 2r assumption for fixed points}
    Let $r\geq 0$. As $n\to \infty$, uniformly over all $\lambda \vdash n$ such that $\lambda_1 = n-r$ and every $\sigma \in \kS_n$ such that $f_1(\sigma) \geq 2r$, we have
    \begin{equation}
        |\RT(\lambda, \alpha)| \lesssim_r M(\sigma)^r,
    \end{equation}
where $\alpha$ is the cycle structure of $\sigma$ written in weakly increasing order.
\end{lemma}
\begin{proof}
The result follows from Lemma~\ref{lem: technical lemma bound on S prime}, since  $\RT(\lambda, \alpha) = \cup_{K \in \ag 0, 2,3,\ldots, r\ad} S'(\lambda, \alpha, K)$.
\end{proof}
We now see how to extend this to all permutations $\sigma$, including those such that “$f_1(\sigma)< 2r$”. Let us add some notation for this section.

Let $n \geq 1$. Let $\alpha = (\alpha_1, \ldots, \alpha_p)$ such that $1\leq \alpha_1 \leq \ldots \leq \alpha_p$ and $\sum_{i=1}^p \alpha_i = n$. Call such an $\alpha$ a \textit{reversed partition} of $n$. Denote the set of all reversed partitions of $n$ by $\mathfrak{R}_n$, and write alternatively $\alpha \vdash_{\textup{rev}} n$ if $\alpha \in \mathfrak{R}_n$.

For us, the reversed partitions $\alpha$ correspond to the cycle structure of a permutation $\sigma$ written in weakly increasing order. Since $M(\sigma)$ depends only on the cycle type of $\sigma$, we may also write $M(\alpha)$ instead. In this section it will be convenient to represent a reversed partition $\alpha$ by $\alpha=1^{f_1}2^{f_2}\ldots$, where $f_j = f_j(\alpha)$ is the number of $\alpha_i$'s equal to $j$. For example we may write $(1,1,3,3,3,3,3, 7,7,7) =  1^2 3^5 7^3 \in \mathfrak{R}_{38}$.

For $\alpha = 1^{f_1}2^{f_2}\ldots \in \mathfrak{R}_n\backslash \ag (1, \ldots, 1) \ad$, denote also $I(\alpha) = \min \ag i\geq 2 \mid f_i \geq 1\ad$.

We now show that the result of Lemma~\ref{lem: character bound finite level general with M sigma with extra 2r assumption for fixed points} holds for reversed partitions such that “$I(\alpha)>r$”.

\begin{lemma}\label{lem: character bound finite level general with M sigma with extra assumption more than r for I alpha}
    Let $r\geq 0$. As $n\to \infty$, uniformly over all $\lambda \vdash n$ such that $\lambda_1 = n-r$ and every $\sigma \in \kS_n$ such that $I(\alpha) >r$, we have
    \begin{equation}
        |\RT(\lambda, \alpha)| \lesssim_r M(\sigma)^r,
    \end{equation}
where $\alpha$ is the cycle structure of $\sigma$ written in weakly increasing order.
\end{lemma}
\begin{proof}
    After placing the ribbons of size 1, there is at most one possibility to place the other ribbons as illustrated in Figure~\ref{fig: ribbon tableaux for large I}.
\begin{figure}[!ht]
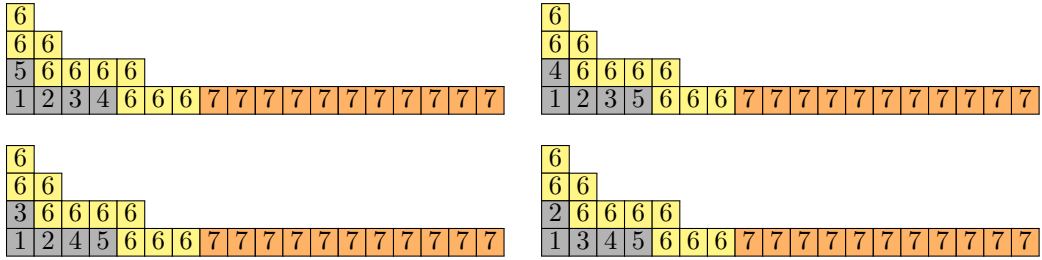

    \centering
{
\ytableausetup{boxsize=1em}

\begin{ytableau}
    \none &*(yellow!60)6 \\
    \none &*(yellow!60)6&*(yellow!60)6 \\
    \none &*(gray!60)5&*(yellow!60)6&*(yellow!60)6&*(yellow!60)6&*(yellow!60)6 \\
    \none &*(gray!60)1&*(gray!60)2&*(gray!60)3&*(gray!60)4&*(yellow!60)6&*(yellow!60)6&*(yellow!60)6&*(orange!60)7&*(orange!60)7&*(orange!60)7&*(orange!60)7&*(orange!60)7&*(orange!60)7&*(orange!60)7&*(orange!60)7&*(orange!60)7&*(orange!60)7&*(orange!60)7\\
    \none 
\end{ytableau}
\begin{ytableau}
    \none &*(yellow!60)6 \\
    \none &*(yellow!60)6&*(yellow!60)6 \\
    \none &*(gray!60)4&*(yellow!60)6&*(yellow!60)6&*(yellow!60)6&*(yellow!60)6 \\
    \none &*(gray!60)1&*(gray!60)2&*(gray!60)3&*(gray!60)5&*(yellow!60)6&*(yellow!60)6&*(yellow!60)6&*(orange!60)7&*(orange!60)7&*(orange!60)7&*(orange!60)7&*(orange!60)7&*(orange!60)7&*(orange!60)7&*(orange!60)7&*(orange!60)7&*(orange!60)7&*(orange!60)7\\
    \none 
\end{ytableau}

\begin{ytableau}
    \none &*(yellow!60)6 \\
    \none &*(yellow!60)6&*(yellow!60)6 \\
    \none &*(gray!60)3&*(yellow!60)6&*(yellow!60)6&*(yellow!60)6&*(yellow!60)6 \\
    \none &*(gray!60)1&*(gray!60)2&*(gray!60)4&*(gray!60)5&*(yellow!60)6&*(yellow!60)6&*(yellow!60)6&*(orange!60)7&*(orange!60)7&*(orange!60)7&*(orange!60)7&*(orange!60)7&*(orange!60)7&*(orange!60)7&*(orange!60)7&*(orange!60)7&*(orange!60)7&*(orange!60)7\\
    \none 
\end{ytableau}
\begin{ytableau}
    \none &*(yellow!60)6 \\
    \none &*(yellow!60)6&*(yellow!60)6 \\
    \none &*(gray!60)2&*(yellow!60)6&*(yellow!60)6&*(yellow!60)6&*(yellow!60)6 \\
    \none &*(gray!60)1&*(gray!60)3&*(gray!60)4&*(gray!60)5&*(yellow!60)6&*(yellow!60)6&*(yellow!60)6&*(orange!60)7&*(orange!60)7&*(orange!60)7&*(orange!60)7&*(orange!60)7&*(orange!60)7&*(orange!60)7&*(orange!60)7&*(orange!60)7&*(orange!60)7&*(orange!60)7\\
    \none 
\end{ytableau}

\begin{ytableau}
    \none 
\end{ytableau}
}
\caption{All possible ribbon coverings of the diagram $ (18,5,2,1)\vdash 26$, for $\alpha = (1,1,1,1,1,10,11)$, and for which the ribbons of size 1 cover the subdiagram $\mu = (4,1)$.
Colors are redundant.}
    \label{fig: ribbon tableaux for large I}
\end{figure}      
    Therefore
    \begin{equation}
        |\RT(\lambda, \alpha)| \leq \sum_{\mu} d_\mu,
    \end{equation}
where the sum is over all diagrams $\mu \vdash f_1(\sigma)$ such that $\mu \subset \lambda$. Since the sum has $O_r(1)$ terms, and each $\mu$ in the sum has at most $r$ boxes over the first row, we deduce that $|\RT(\lambda, \alpha)| \lesssim_r f_1(\sigma)^r \leq M(\sigma)^r$.
\end{proof}

We now want to generalize the result to every permutation. Our strategy is the following. We will show how to explicitly construct a \textit{thinner} version $\beta$ of $\alpha$ which satisfies either $f_1(\beta)\geq 2r$ or $I(\beta)>r$; and such that $|\RT(\lambda, \alpha)|\leq |\RT(\lambda, \beta)|$ but also such that $M(\beta) \lesssim_r M(\alpha)$. This will allow us to conclude.

The idea is that \textit{fragmenting} ribbons into smaller ribbons increases the number of ribbon tableaux. Note that this works if we \textit{fully} fragment a ribbon into ribbons of length 1, but this \textit{does not} work in general (e.g.\ if a ribbon of size 7 is fragmented into a ribbon of size 3 and one of size 4). For example $|\RT((3,1), (4))| = 1$ but $|\RT((3,1), (1,3))| = 0$, and $|\RT((3,2,1), (1,5))| = 1$ but $|\RT((3,2,1), (1,1,4))| = 0$.

\begin{lemma}[Full fragmentation of the smallest ribbon of size more than 1]\label{lem: Full fragmentation of the smallest ribbon of size more than 1}
    Let $n\geq 1$, $\lambda \vdash n$ and $\alpha = 1^{a_1}2^{a_2} \ldots \vdash_{\textup{rev}} n$. Denote $I = I(\alpha)$. Let $b_1 = a_1 + I$, $b_{I} = a_{I} - 1$, and $b_i = a_i$ if $i\notin \ag 1, I\ad$. Then $\beta := 1^{b_1}2^{b_2} \ldots \vdash_{\textup{rev}} n$ satisfies $|\RT(\lambda, \alpha)|\leq |\RT(\lambda, \beta)|$ and $M(\beta) \lesssim_I M(\alpha)$.
\end{lemma}
\begin{proof}
Let $T \in \RT(\lambda, \alpha)$. Denote its ribbons by $\rho_1, \rho_2, \ldots$. For each $i\geq 1$, the $i$-th ribbon is filled with the number $i$, and by definition $\rho_{a_1 + 1}$ is the first ribbon of size $I$. Define a new ribbon tableau $T^*$ as follows:
\begin{itemize}
    \item Start from $T$.
    \item For $i\leq a_1$, let $\rho_i' = \rho_i$. $\rho_i'$ is still filled with the number $i$.
    \item Consider the standard tableau $T'$ of $\rho_{a_1+1}$, filled with the numbers from $a_1 + 1$ to $a_1 + I$, where we first fill the first row from left to right, then the second, and so on. For $i\in \ag a_1 + 1, \ldots, a_1 + I\ad$, let $\rho_i'$ be the ribbon of size 1 filled by the number $i$ in $T'$.
    \item For $i\geq a_1+I+ 1$, let $\rho_{i}'$ be the ribbon $\rho_{i-I+ 1}$, where the content of cells is replaced by $i$ (it was previously $i - I + 1$).
\end{itemize}
We illustrate this process in Figure~\ref{fig: fragmentation process}.
\begin{figure}[!ht]
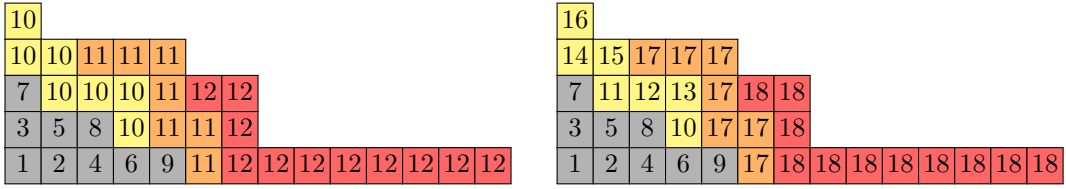

    \centering
{
\ytableausetup{boxsize=1.32em}

\begin{ytableau}
    \none &*(yellow!60)10 \\
    \none &*(yellow!60)10 &*(yellow!60)10&*(orange!60)11&*(orange!60)11&*(orange!60)11 \\
    \none &*(gray!60)7 &*(yellow!60)10&*(yellow!60)10&*(yellow!60)10&*(orange!60)11&*(red!60)12&*(red!60)12 \\
    \none &*(gray!60)3  &*(gray!60)5&*(gray!60)8&*(yellow!60)10&*(orange!60)11&*(orange!60)11&*(red!60)12 \\
    \none &*(gray!60)1&*(gray!60)2&*(gray!60)4&*(gray!60)6&*(gray!60)9&*(orange!60)11&*(red!60)12&*(red!60)12&*(red!60)12&*(red!60)12&*(red!60)12&*(red!60)12&*(red!60)12&*(red!60)12\\
    \none 
\end{ytableau}
\begin{ytableau}
    \none &*(yellow!60)16 \\
    \none &*(yellow!60)14 &*(yellow!60)15&*(orange!60)17&*(orange!60)17&*(orange!60)17 \\
    \none &*(gray!60)7 &*(yellow!60)11&*(yellow!60)12&*(yellow!60)13&*(orange!60)17&*(red!60)18&*(red!60)18 \\
    \none &*(gray!60)3  &*(gray!60)5&*(gray!60)8&*(yellow!60)10&*(orange!60)17&*(orange!60)17&*(red!60)18 \\
    \none &*(gray!60)1&*(gray!60)2&*(gray!60)4&*(gray!60)6&*(gray!60)9&*(orange!60)17&*(red!60)18&*(red!60)18&*(red!60)18&*(red!60)18&*(red!60)18&*(red!60)18&*(red!60)18&*(red!60)18\\
    \none 
\end{ytableau}

\begin{ytableau}
    \none 
\end{ytableau}
}
\caption{Here $\lambda = (14,7,7,5,1)$, $\alpha = 1^9 7^2 11^1$, and $\beta = 1^{16} 7^1 11^1$. On the left is a ribbon tableau $T\in \RT(\lambda, \alpha)$, and on the right is the ribbon tableau $T^*\in \RT(\lambda, \beta)$ obtained by the fragmentation process.
Colors are redundant.}
    \label{fig: fragmentation process}
\end{figure}   
By construction, $T^* \in \RT(\lambda, \beta)$, and the map $T\mapsto T^*$ is injective. It follows that $|\RT(\lambda, \alpha)|\leq |\RT(\lambda, \beta)|$.

Finally, using that $b_i\leq a_i$ for $i\geq 2$ and that $b_1 = a_1 + I$, we obtain
\begin{equation}
    M(\beta) = \max_{i\geq 1} b_i^{1/i} \leq \max(a_1 + I,\max_{i\geq 2}a_i^{1/i})) \leq I +  \max_{i\geq 1}a_i^{1/i} = I + M(\alpha).
\end{equation}
Since by definition we have $M(\alpha) \geq 1$, we conclude that $M(\beta) \lesssim_I M(\alpha)$.
\end{proof}
We now have all the pieces to prove Theorem~\ref{thm: character bound finite level general with M sigma}.

\begin{proof}[Proof of Theorem~\ref{thm: character bound finite level general with M sigma}]
 Let $\lambda \vdash n$ such that $\lambda_1 = n-r$ and let $\sigma \in \kS_n$. By definition we have $M(\sigma)^r = (n^r)^{J(\sigma)}$ and $J(\sigma) \lesssim 1$. Since $d_\lambda \asymp_r n^r$, we obtain that $M(\sigma)^r \asymp_r d_\lambda^{J(\sigma)}$, which proves that the two statements are equivalent.
 
Let us therefore only prove the first one, that is, that $  \bg \ch^\lambda(\sigma) \bd \lesssim_r M(\sigma)^r$.
Denote the cycle type of $\sigma$ written in weakly increasing order by $\alpha = 1^{a_1}2^{a_2}\ldots \vdash_{\textup{rev}} n$.
By the triangle inequality in the Murnaghan--Nakayama rule, we have $\bg \ch^\lambda(\sigma) \bd \leq |\RT(\lambda, \alpha)|$.

Let $m = \sum_{i= 2}^r a_i$. Define $\beta \vdash_{\textup{rev}} n$ from $\alpha$ by applying the fragmentation process iteratively $\min(m,r)$ times to $\alpha$. By Lemma~\ref{lem: Full fragmentation of the smallest ribbon of size more than 1} we have $|\RT(\lambda, \alpha)|\leq |\RT(\lambda, \beta)|$ and $M(\beta) \lesssim_r M(\alpha)$. (There, the implicit constant depends only on $r$ since we have applied the fragmentation process at most $r$ times to values of $I$ which are in $\ag 2, \ldots, r \ad$.)
By construction, if $m \geq r$ then $f_1(\beta) \geq 2r$; and if $m<r$ then $I(\beta) >r$. We deduce from Lemma~\ref{lem: character bound finite level general with M sigma with extra 2r assumption for fixed points} (if $m \geq r$) and Lemma~\ref{lem: character bound finite level general with M sigma with extra assumption more than r for I alpha} (if $m < r$) that $|\RT(\lambda, \beta)| \lesssim_r M(\beta)^r$. Combining everything, we conclude that
\begin{equation}
    \bg \ch^\lambda(\sigma) \bd \leq |\RT(\lambda, \alpha)| \leq |\RT(\lambda, \beta)| \lesssim_r M(\beta)^r \lesssim_r M(\alpha)^r = M(\sigma)^r. \qedhere
\end{equation}
\end{proof}

\begin{proof}[Proof of Theorem~\ref{thm: borne asymptotique r fini pour les caractères peu de points fixes}]
By assumption we have $f_1(\sigma_n) = o(n^{1/2})$ and $f_2(\sigma_n)^{1/2} = o(n^{1/2})$. Moreover for $i\geq 3$, we have $f_i(\sigma_n)^{1/i} \leq n^{1/i} \leq n^{1/3}$. It follows that $M(\sigma_n) = o(n^{1/2})$, and hence that $M(\sigma_n)^r = o(n^{r/2})$. The result then follows from Theorem~\ref{thm: character bound finite level general with M sigma}.
\end{proof}

\begin{proof}[Proof of Corollary~\ref{cor: equivalent asymptotique d lambda fois caractère à une certaine puissance r fini pour les caractères peu de points fixes}]
The result follows from Theorem~\ref{thm: borne asymptotique r fini pour les caractères peu de points fixes},  since if $\lambda \vdash n$ is such that $n-\lambda_1 = r$ we have $d_\lambda \leq n^r$.
\end{proof}

\section{Preliminaries for applications to cutoff profiles}

\subsection{Generalized approximation and comparison lemmas}
The approximation lemma \cite[Lemma~2.1]{Teyssier2020} enabled filtering \textit{low frequencies} to obtain the profile for transpositions, and the comparison lemma for cutoff profiles \cite[Lemma~1.4]{Nestoridi2024comparisonstar} enabled obtaining the profile of a (sequence of) irreducible Markov chain from another one, given that they have similar spectral properties.
In this paper we want to apply the comparison of profiles to chains that may be periodic. These variants also work for time-inhomogeneous chains, for example if we multiply by permutations uniform in a different conjugacy class at each step. This naturally arises in some models of random maps, which are coded by a product of a fixed point free permutation and another permutation, see for instance \cite{Gamburd2006, BudzinskiCurienPetri2019}.

We therefore have to prove a variant of these lemmas. Note however that this section does not contain any new argument, and is included for completeness.

\medskip

In what follows the set of all irreducible representations of a finite group $G$ is denoted by $\widehat{G}$.

\begin{lemma}\label{lem:approximation lemma a bit more general prelim}
    Let $G$ be a finite group. Let $H$ be a subset of $G$, $A\subset B \subset \widehat{G}$, and let $(u_{\alpha})_{\alpha\in \widehat{G}}$ be a family of complex numbers. Then
\begin{equation}\label{eq:approximation lemma}
    \bg \frac{1}{|G|}\sum_{g\in H} \bg \sum_{\alpha \in B} d_\alpha u_\alpha \overline{\ch^\alpha(g)} \bd - \frac{1}{|G|}\sum_{g\in H} \bg \sum_{\alpha \in A} d_\alpha u_\alpha \overline{\ch^\alpha(g)} \bd\bd \leq \sum_{\alpha \in {B\backslash A}} d_\alpha |u_\alpha|.
\end{equation}
\end{lemma}
\begin{proof}
By triangle inequalities, the left hand side of \eqref{eq:approximation lemma} is upper bounded by
\begin{equation}
\begin{split}
  \frac{1}{|G|}\sum_{g\in H} \sum_{\alpha \in B\backslash A} d_\alpha  |u_\alpha| \bg \ch^\alpha(g)\bd = \sum_{\alpha \in B\backslash A} d_\alpha  |u_\alpha| \pg\frac{1}{|G|}\sum_{g\in H} \bg \ch^\alpha(g)\bd \pd. 
\end{split}
\end{equation}
This concludes the proof, since by the Cauchy--Schwarz inequality and orthonormality of characters we have
\begin{equation}
    \frac{1}{|G|}\sum_{g\in H} \bg \ch^\alpha(g)\bd  \leq \frac{1}{|G|}\sum_{g\in G} \bg \ch^\alpha(g)\bd \leq  \frac{1}{|G|} \sqrt{|G| \sum_{g\in G} \bg \ch^\alpha(g)\bd^2} = \frac{1}{|G|} \sqrt{|G|\cdot |G|} = 1. \qedhere
\end{equation}
\end{proof}

Taking $A = \emptyset$ in Lemma~\ref{lem:approximation lemma a bit more general prelim}, we obtain the following.

\begin{lemma}\label{lem: borne DS version sans les carrés ni la racine}
     Let $G$ be a finite group. Let $H\subset G$, $ B \subset \widehat{G}$, and $(u_{\alpha})_{\alpha\in \widehat{G}}$ be a family of complex numbers.
     Then
     \begin{equation}
         \frac{1}{|G|}\sum_{g\in H} \bg \sum_{\alpha \in B} d_\alpha u_\alpha  \overline{\ch^\alpha(g)} \bd \leq \sum_{\alpha \in B} d_\alpha |u_\alpha|.
     \end{equation}
\end{lemma}

Now let us compare these quantities for different families $(u_\alpha)_{\alpha\in \widehat{G}}$.

\begin{lemma}\label{lem: comparaison familles u v}
     Let $G$ be a finite group. Let $H\subset G$, $ B \subset \widehat{G}$, and $(u_{\alpha})_{\alpha\in \widehat{G}}, (v_{\alpha})_{\alpha\in \widehat{G}}$ be two families of complex numbers.
     Then
     \begin{equation}\label{eq: comparaison familles u v}
         \frac{1}{|G|}\sum_{g\in H} \bg \bg \sum_{\alpha \in B} d_\alpha u_\alpha  \overline{\ch^\alpha(g)} \bd -\bg \sum_{\alpha \in B} d_\alpha v_\alpha  \overline{\ch^\alpha(g)} \bd\bd\leq \sum_{\alpha \in B} d_\alpha |u_\alpha-v_\alpha|.
     \end{equation}
\end{lemma}
\begin{proof}
By the triangle inequality, the left hand side of \eqref{eq: comparaison familles u v} is upper bounded by
\begin{equation}
\begin{split}
       \frac{1}{|G|}\sum_{g\in H} \bg \sum_{\alpha \in B} d_\alpha u_\alpha  \overline{\ch^\alpha(g)} -\sum_{\alpha \in B} d_\alpha v_\alpha  \overline{\ch^\alpha(g)} \bd  =  \frac{1}{|G|}\sum_{g\in H} \bg \sum_{\alpha \in B} d_\alpha (u_\alpha - v_\alpha)  \overline{\ch^\alpha(g)}\bd 
\end{split}
\end{equation}
Applying Lemma~\ref{lem: borne DS version sans les carrés ni la racine} to the family $(u_\alpha - v_\alpha)_{\alpha \in \widehat{G}}$ concludes the proof.
\end{proof}

\subsection{Random transpositions in coset distance}

In the random transposition (RT) model on $\kS_n$ introduced by Diaconis and Shahshahani \cite{DiaconisShahshahani1981}, the two hands pick cards independently, so the probability that they are the same is $1/n$, and the model has inherent laziness $1/n$, which makes the walk aperiodic. This walk is driven by the measure $\frac{1}{n}\delta_{\Id} + \pg 1- \frac{1}{n} \pd \Unif_{\cT}$, where we recall that $\cT = \cT^{(n)} \subset \kS_n$ is the conjugacy class of transpositions.

The pure transposition walk (pRT) is the walk driven by  $\mu_{\pRT} = \Unif_{\cT}$. This corresponds to taking at each step two distinct cards and swapping them.

\medskip

The eigenvalues of pRT are exactly the characters $\chi^\lambda(\cT)$ (for $\lambda \vdash n$), and therefore have good symmetry properties: we have $\chi^{\lambda'}(\cT) = - \chi^\lambda(\cT)$, where $\lambda'$ is the conjugate diagram of~$\lambda$.
On the other hand, the eigenvalues of RT are $\frac{1}{n} + \pg 1- \frac{1}{n} \pd \chi^\lambda(\tau)$, and laziness then removes some symmetry for the eigenvalues. While it affects the eigenvalues that are close to 1 (that is, for $\lambda$ such that $\lambda_1 = n-O(1)$) only by $O(1/n^2)$, which is negligible, it affects the eigenvalues close to $-1$ (that is, for $\lambda$ such that $\lambda_1' = n-O(1)$) by $2/n + O(1/n^2)$, which makes all these representations asymptotically negligible close to the mixing time in the distance to stationarity.

\medskip

The two models, however, have the same asymptotic behaviour. The number of fixed points of a permutation taken uniformly in either $\kS_n$, or $\kA_n$, or $\kS_n \backslash\kA_n$, is asymptotically distributed as $\Poiss(1)$. At a computational level, if $\lambda$ is such that $\lambda_1 =n-O(1)$, the contribution of $\lambda$ for RT is (asymptotically) the same as the contribution of $\lambda$ and $\lambda'$ combined for pRT, and $\lambda'$ does not contribute to RT. All other estimates are the same. Making these minor adaptations in the proof of \cite{Teyssier2020} immediately gives the following for pRT.

\begin{theorem}\label{thm: profile pure transpositions}
     Let $a\in \bbR$. Let $(t_{n})$ be a sequence of integers such that $t_{n} = \frac{1}{2}n (\ln n + a + o(1))$. Then as $n \to \infty$, we have
\begin{equation}
       \dtv(\Unif_{\cT^{(n)}}^{*t_n}, \Unif_{\kE({\cT^{(n)}},t_n)})  \to \dtv\pg \Poiss\pg 1+e^{-a}\pd,  \Poiss(1) \pd.
    \end{equation}
\end{theorem}

\section{Cutoff profiles for conjugacy invariant random walks}\label{s: cutoff profiles conjugacy classes}
\subsection{Cutoff profile for conjugacy classes}
\subsubsection{Setup}
Let $n\geq 2$, $\cC \in \Conj^*(\mathfrak{S}_n)$ and $t\geq 0$. Recall that 
\begin{equation}
    \mathfrak{E}(\cC, t) := \begin{cases}
        \mathfrak{S}_n \backslash\mathfrak{A}_n & \text{ if } \cC \subset \mathfrak{S}_n \backslash\mathfrak{A}_n \text{ and } t \text{ is odd},\\
        \mathfrak{A}_n & \text{ otherwise.}
    \end{cases}
\end{equation}
Define $\widehat{\kS_n}^{**} = \widehat{\kS_n} \backslash\ag [n], [1^n] \ad$. The (total variation) distance to coset stationarity for the $\Unif_{\cC}$-walk after $t$ steps can be written as:
\begin{equation}\label{eq: def distance to coset stationarity}
\begin{split}
    \dtv(\Unif_{\cC}^{*t}, \Unif_{\mathfrak{E}(\cC, t)}) & = \frac{1}{2} \sum_{\sigma \in \mathfrak{E}(\cC,t)} \bg \Unif_{\cC}^{*t}(\sigma) - \frac{2}{n!} \bd \\
    & = \frac{1}{2}\sum_{\sigma \in \mathfrak{E}(\cC,t)} \bg \sum_{\lambda \in \widehat{\mathfrak{S}_n}^{**}} \frac{d_\lambda}{n!} \chi^\lambda(\cC)^t \ch^\lambda(\sigma)\bd,
\end{split}
\end{equation}
where the last equality is obtained by applying the inverse Fourier transform.

\medskip

For $\lambda \in \widehat{\kS_n}$, denote $\rsym(\lambda) = \min(n-\lambda_1, n-\lambda_1')$.

For $R\geq 1$, denote $A(R) = \ag \lambda\in \widehat{\kS_n} \mid 1\leq \rsym(\lambda) \leq R \ad$.

For $r\geq 0$, let $\zeta(r) = \ln \max(1, r^{1/8})$, as in \cite[Section~5.5]{Olesker-TaylorTeyssierThevenin2025sharpboundsandcutoff}. We will only use that $\zeta(r)\to \infty$ as $r\to \infty$.

\subsubsection{Proof of Theorem~\ref{thm: profil TV pour les CC}}
\begin{proof}[Proof of Theorem~\ref{thm: profil TV pour les CC}]
Let $(s_n)_{n\in S}$ be a sequence of integers such that:
\begin{itemize}
    \item $s_n$ has the same parity as $t_n$ if $\cC^{(n)}$ is odd, and $s_n$ is even if $\cC^{(n)}$ is even;
    \item $s_n = \frac{1}{2}n(\ln n + a + o(1))$ as $n\to \infty$.
\end{itemize}

Recalling \eqref{eq: def distance to coset stationarity} and applying Lemma~\ref{lem: comparaison familles u v} with
$B = \widehat{\kS_n}^{**}$, $u_\lambda =\chi^\lambda(\cC^{(n)})^{t_n}$, and $v_\lambda =\chi^\lambda(\cT^{(n)})^{s_n}$, we obtain
\begin{equation}
\begin{split}
    & 2\bg\dtv(\Unif_{\cC^{(n)}}^{*t_n}, \Unif_{\mathfrak{E}(\cC^{(n)}, t_n)}) - \dtv(\Unif_{\cT^{(n)}}^{*s_n}, \Unif_{\mathfrak{E}(\cT^{(n)}, s_n)})\bd \\
    \leq \; & \sum_{\lambda\in \widehat{\kS_n}^{**}} d_\lambda \bg \chi^\lambda(\cC^{(n)})^{t_n} - \chi^\lambda(\cT^{(n)})^{s_n}\bd.
\end{split}
\end{equation}
Let $\varepsilon\in (0,1)$, and let $R\geq 1$ be an integer. 
By the triangle inequality and \cite[Theorem~5.8]{Olesker-TaylorTeyssierThevenin2025sharpboundsandcutoff}, as $S \ni n\to \infty$ we have
\begin{equation}\label{eq: bound character CC from OTTT form with zeta}
\sum_{\lambda\in \widehat{\kS_n}^{**}\backslash A(R)} d_\lambda \bg \chi^\lambda(\cC^{(n)})^{t_n} - \chi^\lambda(\cT^{(n)})^{s_n}\bd\leq 2 \sum_{\lambda\in \widehat{\kS_n}^{**}\backslash A(R)} d_\lambda^{(a-\zeta(R) + o(1))/\ln n}. 
\end{equation}
Since $\zeta(r)\to \infty$ as $r\to \infty$, it follows from \cite[Theorem~1.8]{TeyssierThevenin2025virtualdegreeswitten} that there exists $R_0$ such that for all $n$ sufficiently large we have
\begin{equation}
    \sum_{\lambda\in \widehat{\kS_n}^{**}\backslash A(R_0)} d_\lambda \bg \chi^\lambda(\cC^{(n)})^{t_n} - \chi^\lambda(\cT^{(n)})^{s_n}\bd\leq \varepsilon.
\end{equation}
Moreover, by Theorem~\ref{thm: equivalent asymptotique r fini pour les caractères suffisamment de points fixes} (note that the definition of $s_n$ ensures that $\chi^\lambda(\cC^{(n)})^{t_n}$ and $\chi^\lambda(\cT^{(n)})^{s_n}$ do not have different signs, and hence that they compensate, see Remark~\ref{rk: sign change if transpose and odd}), we have as $S \ni n \to \infty$
\begin{equation}
    \sum_{\lambda\in  A(R_0)} d_\lambda \bg \chi^\lambda(\cC^{(n)})^{t_n} - \chi^\lambda(\cT^{(n)})^{s_n}\bd \to 0.
\end{equation}
Since $\varepsilon$ was arbitrary, we have proved that as $S \ni n \to \infty$
\begin{equation}
    \bg\dtv(\Unif_{\cC^{(n)}}^{*t_n}, \Unif_{\mathfrak{E}(\cC^{(n)}, t_n)}) - \dtv(\Unif_{\cT^{(n)}}^{*s_n}, \Unif_{\mathfrak{E}(\cT^{(n)}, s_n)})\bd \to 0.
\end{equation}
The desired result then follows from Theorem~\ref{thm: profile pure transpositions}.
\end{proof}

\subsubsection{Estimates in other distances}
Proceeding as above, we can obtain estimates for the $L^2$ and $L^\infty$ distances. Given an integer $n\geq 2$, a conjugacy class $\ag \Id\ad \ne \cC \subset \kS_n$, and an integer $t\geq 0$, the $L^2$ distance to stationarity for the $\Unif_{\cC}$-walk at time $t$ is given, denoting the probability to go from $x\in \kS_n$ to $y\in \kS_n$ in $s$ steps by $p_s(x,y)$ and setting $A = |\kE(\cC,t)| = n!/2$, by
\begin{equation}
 \textup{d}_2^{\cC}(t) := \sqrt{ \frac{1}{A} \sum_{\sigma \in \kE(\cC, t) } (Ap_t(\Id, \sigma) - 1)^2} = \sqrt{\frac{1}{2}\sum_{\lambda \in \widehat{\kS_n}^{**}} d_\lambda^2 \bg \chi^\lambda(\cC^{(n)}) \bd^{2t}}.
\end{equation}
and its $L^\infty$ distance to stationarity at time $2t$ is given by
\begin{equation}
 \textup{d}_\infty^{\cC}(2t) := A\max_{\sigma \in \kA_n} p_{2t}(\Id, \sigma) - 1 = A p_{2t}(\Id, \Id) - 1 =  \frac{1}{2}\sum_{\lambda \in \widehat{\kS_n}^{**}} d_\lambda^2 \bg \chi^\lambda(\cC^{(n)}) \bd^{2t}.
\end{equation}
There, the penultimate equality comes from the facts that the walk is a simple random walk on a vertex-transitive graph, and hence the most likely vertex after an even number of steps is the starting vertex; and the last equality comes from the inverse Fourier transform.

\begin{proposition}
Let $\delta>0$. For each $n$, let $ \cC^{(n)} \in\Conj^*(\kS_n)$. Assume that $f_1(\cC^{(n)})\geq \delta n$ for $n$ large enough. Let $a\in \bbR$.
Let $(t_n)$ be a sequence of positive real numbers such that $t_n = \frac{\ln n + a + o(1)}{\ln(n/f_1(\cC^{(n)}))}$. Let $\theta >0$.   
As $n\to \infty$, we have
\begin{equation}
   \frac{1}{2} \sum_{\lambda \in \widehat{\kS_n}^{**}} \pg d_\lambda \bg \chi^\lambda(\cC^{(n)}) \bd^{t_n}\pd^{\theta} \to \sum_{r\geq 1} e^{-ar\theta} \sum_{\mu \vdash r} \pg\frac{d_\mu}{r!}\pd^\theta.
\end{equation}
In particular, as $n\to \infty$
\begin{equation}
 \frac{1}{2}\sum_{\lambda \in \widehat{\kS_n}^{**}} d_\lambda^2 \bg \chi^\lambda(\cC^{(n)}) \bd^{2t_n} \to e^{e^{-2a}}-1.
\end{equation}
\end{proposition}
\begin{proof}
Assume now that $a\geq 0$. Let $\varepsilon>0$. Proceeding as in the proof of Theorem~\ref{thm: profil TV pour les CC}, there exists $R_0$ (which depends on $\varepsilon$, $\delta$, and $\theta$) such that for $n$ large enough we have
    \begin{equation}
        \sum_{\lambda \in \widehat{\kS_n}^{**}\backslash A(R_0)} \pg d_\lambda \bg \chi^\lambda(\cC^{(n)}) \bd^{t_n}\pd^{\theta} \leq \varepsilon.
    \end{equation}
Moreover, as $n\to \infty$, uniformly over all $\lambda \in A(R_0)$, by Corollary~\ref{cor: equivalent asymptotique d lambda fois caractère à une certaine puissance r fini pour les caractères suffisamment de points fixes} we have as $n\to \infty$,
\begin{equation}
    \pg d_\lambda \bg \chi^\lambda(\cC^{(n)}) \bd^{t_n}\pd^{\theta} \to  \pg\frac{d_{\lambda^*}}{r!}\pd^{\theta} e^{-ar\theta}.
\end{equation}
Since $\varepsilon$ was arbitrary, it follows as in the proof of Theorem~\ref{thm: profil TV pour les CC} that as $n\to \infty$
\begin{equation}
    \sum_{\lambda \in \widehat{\kS_n}^{**}} \pg d_\lambda \bg \chi^\lambda(\cC^{(n)}) \bd^{t_n}\pd^{\theta} \to 2\sum_{r\geq 1} e^{-ar\theta} \sum_{\mu \vdash r} \pg\frac{d_\mu}{r!}\pd^\theta,
\end{equation}
where the sum above is convergent since for any $r\geq 1$,
\begin{equation}
    e^{-ar\theta} \sum_{\mu \vdash r} \pg\frac{d_\mu}{r!}\pd^\theta \leq  e^{-ar\theta} p(r)\max_{\mu \vdash r} \pg\frac{d_\mu}{r!}\pd^\theta \leq  e^{-ar\theta} p(r) \pg\frac{\sqrt{r!}}{r!}\pd^\theta  = \frac{e^{-ar\theta} p(r)}{\sqrt{r!}}.
\end{equation}
In particular, if $\theta = 2$, we have 
\begin{equation}
     \sum_{\lambda \in \widehat{\kS_n}^{**}} \pg d_\lambda \bg \chi^\lambda(\cC^{(n)}) \bd^{t_n}\pd^{\theta} \to 2\sum_{r\geq 1} \frac{e^{-2ar}}{r!} = 2\pg e^{e^{-2a}}-1 \pd. \qedhere
\end{equation}
\end{proof}
We immediately deduce the shapes of the $L^2$ and $L^{\infty}$ profiles.
\begin{corollary}
Let $\delta>0$. For each $n$, let $ \cC^{(n)} \in\Conj^*(\kS_n)$. Assume that $f_1(\cC^{(n)})\geq \delta n$ for $n$ large enough. Let $a\in \bbR$.
Let $S\subset \bbN$ and $(t_n)_{n\in S}$ be a sequence of positive integers such that, along $S$, $t_n = \frac{\ln n + a + o(1)}{\ln(n/f_1(\cC^{(n)}))}$. Let $\theta >0$.   
As $S \ni n\to \infty$, we have
\begin{equation}
    \textup{d}_2^{\cC^{(n)}}(t_n) \to \sqrt{e^{e^{-2a}}-1} \quad \quad \text{ and } \quad \quad \textup{d}_\infty^{\cC^{(n)}}(2t_n) \to e^{e^{-2a}}-1.
\end{equation}
\end{corollary}

\subsection{Cutoff and cutoff profile for random involutions}
The goal of this section is to prove Theorem~\ref{thm: profil pour les involutions aléatoires}. In this section we omit the dependence on $p\in (0,1)$ from asymptotic notation, and write for example $O(\cdot)$ for $O_p(\cdot)$.

\subsubsection{Some intuition}
 Let $p\in (0,1)$. Let $n$ be a large even integer. Recall from \eqref{eq: RI def} that the random involution walk on $\kS_n$ with parameter $p$ is the walk driven by the measure
\begin{equation}
    \RI_{n,p} = \sum_{s=0}^{n/2} \binom{n/2}{s}(1-p)^{s}p^{n/2-s} \Inv_{n,s},
\end{equation}
where $\Inv_{n,s}$ is the uniform measure on permutations of $\kS_n$ with cycle type $(2^{s}, 1^{n-2s})$.
By definition this is a class function, and the resulting Markov chain is aperiodic.

\medskip

 When one picks a permutation according to $\RI_{n,p}$, one picks a conjugacy invariant permutation $\sigma$ whose support $\supp(\sigma)$ follows approximately a normal law of average $(1-p)n$ and standard deviation $\Theta(\sqrt{n})$, so $\supp(\sigma)$ is \textit{very} concentrated around $(1-p)n$. To prove a cutoff it would be enough that typically $\supp(\sigma) = (1-p)n (1+o(1))$, and to prove the profile we need that typically $\supp(\sigma) = (1-p)n (1+o(1/\ln n))$.

\medskip

To write this rigorously we couple random involutions with a truncated version of random involutions, for which the support is \textit{always} in an interval around $(1-p)n$, and then estimate the eigenvalues by averaging, which gives essentially the same estimates as for conjugacy classes.

\subsubsection{Coupling with truncated random involutions (TRI)}
First, we define a truncated binomial distribution. Denote the probability measure $\Bin(n/2,1-p)$ by $\varphi$, set 
\begin{equation}
    m = n^{4/5},
\end{equation}
and let
\begin{equation}
    I = I_{n,p} = [(1-p)n/2 -m, (1-p)n/2 +m]\cap \bbZ.
\end{equation}
For $k\in I$, set 
\begin{equation}
    \varphi_{\textup{trunc}}(k) = \frac{\varphi(k)}{\varphi(I)}.
\end{equation}
In words, $\varphi_{\textup{trunc}}$ is a rescaled truncation of $\varphi$ on a small interval of size $\asymp m = n^{4/5}$ close to its average.
We have $\mathrm{Var}(\varphi) \asymp n$ so $1-\varphi(I) =O(n/m^2) = O(n^{-3/5})$ by the Bienaymé--Chebyshev inequality. (Much better bounds can be proved but this is not necessary here.)
It follows that
\begin{equation}\label{eq: bound TV phi psi}
    \dtv(\varphi,\varphi_{\textup{trunc}}) = O(n^{-3/5}).
\end{equation}

Let $M \sim \varphi$ and $N\sim \varphi_{\textup{trunc}}$. Let $\sigma$ (resp. $\tau$) be a uniform involution with support $2M$ (resp. $2N$). The law of $\sigma$ is $\RI_{n,p}$, and the law of $\tau$ is
\begin{equation}
    \TRI_{n,p} = \frac{1}{\varphi(I_{n,p})}\sum_{s\in I_{n,p}} \binom{n/2}{s}(1-p)^{s}p^{n/2-s} \Inv_{n,s},
\end{equation}
By construction, we have $ \dtv(\RI_{n,p}, \TRI_{n,p}) \leq \dtv(\varphi,\varphi_{\textup{trunc}})$, so using \eqref{eq: bound TV phi psi} we obtain
\begin{equation}
    \dtv(\RI_{n,p}, \TRI_{n,p}) = O(n^{-3/5}).
\end{equation}
A union bound immediately implies the following.

\begin{lemma}\label{lem: coupling RI TRI}
   Let $p\in (0,1)$. As $n\to \infty$, uniformly over all $0\leq t \leq 10 \log_{1/p}n$, we have
   \begin{equation}
    \dtv(\RI_{n,p}^{*t}, \TRI_{n,p}^{*t}) = O\pg \frac{\ln n}{n^{3/5}}\pd = o(1).
\end{equation}
\end{lemma}
As a result, the asymptotics for the total variation distance to stationarity for $\RI_{n,p}$ and $\TRI_{n,p}$ are the same, and we only need to study the second one.

\subsubsection{Proof of Theorem~\ref{thm: profil pour les involutions aléatoires}}
As we will see, the proof of Theorem~\ref{thm: profil pour les involutions aléatoires} is similar to that of Theorem~\ref{thm: profil TV pour les CC}. The first difference is that the chain is aperiodic so we can compare it to the original random transposition chain (which has laziness $1/n$), as done in \cite{Nestoridi2024comparisonstar}, rather than with pure transpositions. The second difference is that the eigenvalues are a mixture of eigenvalues of walks driven by the uniform measures on conjugacy classes, so some averaging needs to be done.
\begin{proof}[Proof of Theorem~\ref{thm: profil pour les involutions aléatoires}]
     Let $(s_n)_{n\in S}$ be a sequence of integers such that $s_n = \frac{1}{2} n ((\ln n) + a + o(1))$.  

By conjugacy invariance, the eigenvalues of the $\TRI_{n,p}$-walk are, for $\lambda \vdash n$,
\begin{equation}
    \alpha_\lambda := \bbE_{\tau \sim \TRI_{n,p}} \cg \chi^\lambda(\tau)\cd = \frac{1}{\varphi(I_{n,p})}\sum_{s\in I_{n,p}} \binom{n/2}{s}(1-p)^{s}p^{n/2-s} \chi^\lambda(2^s, 1^{n-2s}).
\end{equation}
Let $\T_n = \frac{1}{n} + \pg 1-\frac{1}{n} \pd \Unif_{\cT_n}$ be the driving measure of the transposition walk (with laziness $1/n$). Denote the eigenvalues of the $\T_n$-walk by, for $\lambda \vdash n$, $\beta_\lambda := \frac{1}{n} + \pg 1-\frac{1}{n} \pd\chi^\lambda(2, 1^{n-2})$.

By Lemma~\ref{lem: coupling RI TRI} and applying Lemma~\ref{lem: comparaison familles u v} with
$B = \widehat{\kS_n}^{*} = \widehat{\kS_n}\backslash\ag [n] \ad$, $u_\lambda =\alpha_\lambda^{t_n}$, and $v_\lambda = \beta_\lambda^{s_n}$, we obtain
\begin{equation}
\begin{split}
    & 2\bg\dtv(\RI_{n,p}^{*t_n}, \Unif_{\kS_n}) - \dtv(\T_n^{*s_n}, \Unif_{\kS_n})\bd +o(1) \\
    = \; & 
    2\bg\dtv(\TRI_{n,p}^{*t_n}, \Unif_{\kS_n}) - \dtv(\T_n^{*s_n}, \Unif_{\kS_n})\bd
    \leq \sum_{\lambda\in \widehat{\kS_n}^{*}} d_\lambda \bg \alpha_\lambda^{t_n} - \beta_\lambda^{s_n}\bd.
\end{split}
\end{equation}     
Let $\varepsilon\in (0,1)$, and let $R\geq 1$ be an integer. 
By \cite[Theorem~5.8]{Olesker-TaylorTeyssierThevenin2025sharpboundsandcutoff}, we have 
\begin{equation}\label{eq: bound characters bulk}
d_\lambda\bg \alpha_\lambda \bd^{t_n}
\leq \max_{\sigma \in \kS_n \du |f_1(\sigma) -pn| \leq m}d_\lambda\bg \chi^\lambda(\sigma)\bd^{t_n} \leq  d_\lambda^{(a-\zeta(R) + o(1))/\ln n},
\end{equation}
so we obtain, exactly as in the proof of Theorem~\ref{thm: profil TV pour les CC}, that 
\begin{equation}
    \sum_{\lambda\in \widehat{\kS_n}^{*} \du \rsym(\lambda)>R} d_\lambda\bg \alpha_\lambda\bd^{t_n} \leq \sum_{\lambda\in \widehat{\kS_n}^{*} \du \rsym(\lambda)>R} d_\lambda^{(a-\zeta(R) + o(1))/\ln n}.
\end{equation}
Therefore, by the triangle inequality and using also \cite[Lemma~4.1]{Teyssier2020} to bound the sum for lazy transpositions, there exists $R_0$ such that for all $n$ large enough we have
\begin{equation}
\begin{split}
 & \sum_{\lambda\in \widehat{\kS_n}^{*} \du \rsym(\lambda)>R_0}d_\lambda \bg \alpha_\lambda^{t_n} - \beta_\lambda^{s_n}\bd \\
 \leq \; & \sum_{\lambda\in \widehat{\kS_n}^{*} \du \rsym(\lambda)>R_0}d_\lambda \bg \alpha_\lambda\bd^{t_n}  + \sum_{\lambda\in \widehat{\kS_n}^{*} \du \rsym(\lambda)>R_0}d_\lambda \bg\beta_\lambda\bd^{s_n}  \\ 
 \leq \; & \varepsilon.
\end{split}
\end{equation}
Now denote
\begin{equation}
    A'(R_0) := \ag \lambda\in \widehat{\kS_n} \mid 1\leq r(\lambda) \leq R_0 \ad = \ag \lambda\in \widehat{\kS_n} \mid 1\leq n-\lambda_1 \leq R_0 \ad,
\end{equation}
and
\begin{equation}
    A''(R_0) := \ag \lambda\in \widehat{\kS_n} \mid 1\leq r(\lambda') \leq R_0 \ad = \ag \lambda\in \widehat{\kS_n} \mid 0\leq n-\lambda_1' \leq R_0 \ad.
\end{equation}
Averaging Theorem~\ref{thm: character estimate many fixed points with error term} we obtain that uniformly over $\lambda \in A'(R_0)$, we have 
\begin{equation}
    \alpha_\lambda \sim p^{r(\lambda)},
\end{equation}
and uniformly over $\lambda \in  A''(R_0)$, we have 
\begin{equation}
    \alpha_\lambda = (1/2 + o(1))p^{r(\lambda)} - (1/2 + o(1))p^{r(\lambda)} = o(1),
\end{equation}
since a permutation picked according to $\TRI_{n,p}$ has probability $1/2 + o(1)$ of being even so there are sign compensations (see Remark~\ref{rk: sign change if transpose and odd}). 
It is also known from \cite{Teyssier2020} that $\sum_{\lambda \in A''(R_0)} d_\lambda \beta_\lambda^{s_n} = o(1)$.
We deduce that
\begin{equation}
    \sum_{\lambda\in A''(R_0)} d_\lambda \bg \alpha_\lambda^{t_n} - \beta_\lambda^{s_n}\bd = o(1)
\end{equation}
and
\begin{equation}
\begin{split}
      \sum_{\lambda\in A'(R_0)} d_\lambda \bg \alpha_\lambda^{t_n} - \beta_\lambda^{s_n}\bd & = \sum_{\lambda\in A'(R_0)} d_\lambda \bg \frac{e^{-ar +o(1)}}{n^r} - \frac{e^{-ar +o(1)}}{n^r}\bd \\ & = \sum_{\lambda\in A'(R_0)} d_\lambda  o(n^{-r}) = o(1).
\end{split}  
\end{equation}
Putting everything together, and since $\varepsilon$ was arbitrary, we obtain
\begin{equation}
    \bg\dtv(\RI_{n,p}^{*t_n}, \Unif_{\kS_n}) - \dtv(\T_n^{*s_n}, \Unif_{\kS_n})\bd  = o(1).
\end{equation}
Plugging \cite[Theorem~1.1]{Teyssier2020} into this concludes the proof.
\end{proof}

\section{Numerics for transpositions on a deck of 52 cards}
Let $n\geq 2$ and $t\geq 0$. For $r\geq 1$, denote
\begin{equation}
    M_r^{(n)}(t) = \frac{1}{2}\sum_{\sigma \in \mathfrak{E}(\cT,t)} \bg \sum_{\lambda \in \widehat{\mathfrak{S}_n}^{**} \du 1\leq \rsym(\lambda) \leq r }\frac{d_\lambda}{n!} \chi^\lambda(\cT)^t \ch^\lambda(\sigma)\bd.
\end{equation}
and
\begin{equation}
    E_r^{(n)}(t) =  \frac{1}{2} \sum_{\lambda \in \widehat{\mathfrak{S}_n}^{**} \du \rsym(\lambda) = r} d_\lambda \bg \chi^\lambda(\cT)\bd^t.
\end{equation}
Denote also $\textup{d}^{(n)}(t) = \dtv(\Unif_{\cT}^{*t}, \Unif_{\mathfrak{E}(\cT, t)})$.
Then by \eqref{eq: def distance to coset stationarity} and Lemma~\ref{lem:approximation lemma a bit more general prelim}, we have 
\begin{equation}\label{eq: distance rewritten with M and E}
    \bg \textup{d}^{(n)}(t) - M_1^{(n)}(t)\bd \leq \sum_{r\geq 2} E_r^{(n)}(t).
\end{equation}
We will first bound the error terms $\sum_{r\geq 2} E_r^{(n)}(t)$ and estimate the main term $M_1^{(n)}(t)$ for any value of $n$ and $t$, and then provide numerics for $n=52$.

\subsection{Bound on the contribution of most representations}

We will use the following classical eigenvalue estimate.
\begin{lemma}\label{lem:classical bound on eigenvalues for transpositions}
    Let $n\geq 2$ be even. Let $\lambda \vdash n$. The following holds.
\begin{enumerate}
    \item If $1\leq \rsym(\lambda)\leq n/2$, then $|\chi^\lambda(\cT)| \leq 1 - \frac{2 \rsym(\lambda)(n-\rsym(\lambda)+1)}{n(n-1)}$.
    \item If $\rsym(\lambda)\geq n/2$, then $|\chi^\lambda(\cT)| \leq 1/2$.
\end{enumerate}
\end{lemma}
\begin{proof}
    This follows from \cite[Lemma~3.D.2]{LivreDiaconis1988} and symmetry. (We emphasize that in \cite{LivreDiaconis1988}, the notation $r(\lambda)$ means $\chi^\lambda(\cT)$, while for us $r(\lambda) = n-\lambda_1$.)
\end{proof}
\begin{lemma}\label{lem: sumple bound on sums of dimensions}
    Let $n\geq 2$. Let $r\geq 1$. We have
    \begin{equation}
        \sum_{\lambda \vdash n \du r(\lambda) = r} d_\lambda \leq n^r.
    \end{equation}
\end{lemma}
\begin{proof}
For $\lambda \vdash n$ such that $r(\lambda) = r$, we have $d_\lambda \leq \binom{n}{r}d_{\lambda^*} \leq n^r \frac{d_{\lambda^*}}{r!} \leq n^r \frac{(d_{\lambda^*})^2}{r!}$. The result then follows from the identity $\sum_{\mu \vdash r} d_\mu^2 = r!$.
\end{proof}
Denote the number of integer partitions of an integer $n$ by $p(n)$. 
\begin{lemma}\label{lem: bound on E r for r at least 2}
    Let $n\geq 2$. For any $t\geq 0$, we have
    \begin{equation}
        \sum_{r\geq 2} E_r^{(n)}(t) \leq \pg \sum_{r=2}^{\lf (n-1)/2\rf} n^r \pg 1 - \frac{2r(n-r+1)}{n(n-1)}\pd^{t} \pd + \frac{\sqrt{p(n)n!}}{2^{t+1}}.
    \end{equation}
\end{lemma}
\begin{proof}
    By symmetry, Lemma~\ref{lem: sumple bound on sums of dimensions}, and Lemma~\ref{lem:classical bound on eigenvalues for transpositions} (a), we have
\begin{equation}
\begin{split}
      \sum_{r = 2}^{\lf (n-1)/2\rf} E_r^{(n)}(t) & \leq \sum_{r=2}^{\lf (n-1)/2\rf} n^r \pg 1 - \frac{2r(n-r+1)}{n(n-1)}\pd^{t}.
\end{split}
\end{equation}
By Lemma~\ref{lem:classical bound on eigenvalues for transpositions} (b), the Cauchy--Schwarz inequality, and using that $\sum_{\lambda\vdash n}d_\lambda^2 = n!$, we have
\begin{equation}
    \frac{1}{2}\sum_{r\geq \lf (n+1)/2\rf}\sum_{\lambda \vdash n \du \rsym(\lambda) = r} d_\lambda |\chi^\lambda(\cT)|^t \leq  \frac{1}{2}\sum_{\lambda \vdash n} d_\lambda (1/2)^t \leq (1/2)^{t+1}\sqrt{p(n)n!}.
\end{equation}
The result follows.
\end{proof}

\subsection{Contribution of representations with (symmetrized) level 1}

Let us now study the contribution of the representations $\lambda$ with $\rsym(\lambda) = 1$. For these representations, we have to make an estimate which is both a lower and an upper bound.

We first show that the probability that a random permutation picked in $\kA_n$ or $\kS_n \backslash \kA_n$ has no fixed point is extremely close to $1/e$. The use of a determinant to count derangements below is standard, and the fact that the number of fixed points of a random permutation is very close in distribution to $\Poiss(1)$ was emphasized in \cite{DiaconisMiclo2023fixedpoints}.

\begin{lemma}\label{lem: borne factorielle permutations paires et impaires sans point fixe}
    Let $n\geq 2$. Let $X_n^{\textup{even}} \sim \Unif_{\kA_n}$ and $X_n^{\textup{odd}} \sim \Unif_{\kS_n\backslash\kA_n}$.  Then
\begin{equation}
    \bg \bbP(f_1(X_n^{\textup{even}}) = 0)-1/e\bd \leq \frac{1}{(n-1)!} \quad \text{ and } \quad \bg \bbP(f_1(X_n^{\textup{odd}}) = 0)-1/e\bd \leq \frac{1}{(n-1)!}.
\end{equation}
\end{lemma}
\begin{proof}
Denote the number of derangements (i.e.\ fixed point free permutations) of $\kS_n$ by $D_n$. Denote the number of even and odd derangements of $\kS_n$ (i.e.\ of derangements that are in $\kA_n$ or $\kS_n \backslash \kA_n$) by $D_n^{\textup{even}}$ and $D_n^{\textup{odd}}$.
It is standard that
\begin{equation}
    D_n^{\textup{even}} + D_n^{\textup{odd}} = D_n = n!\sum_{i=0}^n\frac{(-1)^i}{i!} \quad \text{ and } \quad D_n^{\textup{even}} - D_n^{\textup{odd}} = (-1)^{n-1}(n-1).
\end{equation}
($D_n^{\textup{even}} - D_n^{\textup{odd}}$ is the determinant of the $n$ by $n$ matrix with 0's on the diagonal and 1's elsewhere.)
Hence we have 
\begin{equation}
    \bbP(f_1(X_n^{\textup{even}})=0) = \frac{2}{n!}D_n^{\textup{even}} = \sum_{i=0}^n\frac{(-1)^i}{i!} + \frac{(-1)^{n-1}(n-1)}{n!}
\end{equation}
and therefore
\begin{equation}
    \bg \bbP(f_1(X_n^{\textup{even}})=0) - 1/e\bd = \bg \bbP(f_1(X_n^{\textup{even}})=0) - \sum_{i=0}^\infty \frac{(-1)^i}{i!}\bd \leq \bg \sum_{i=n+1}^{\infty}\frac{(-1)^i}{i!}\bd  + \frac{(n-1)}{n!},
\end{equation}
so we finally obtain
\begin{equation}
    \bg \bbP(f_1(X_n^{\textup{even}})=0) - 1/e\bd \leq \frac{1}{(n+1)!} + \frac{(n-1)(n+1)}{(n+1)!} =\frac{n^2}{(n+1)!}\leq \frac{1}{(n-1)!}.
\end{equation}
The same arguments also give $\bbP(f_1(X_n^{\textup{odd}})=0) \leq 1/(n-1)!$.
\end{proof}

\begin{lemma}\label{lem: factorial approximation for the first representation and its conjugate}
    Let $n \geq 3$ and $t\geq 0$. Then
    \begin{equation}
        \bg M_1^{(n)}(t) - \frac{(n-1)}{e}\pg \frac{n-3}{n-1} \pd^t  \bd \leq \frac{1}{(n-2)!}.
    \end{equation}
\end{lemma}
\begin{proof}
    For $\sigma \in \kE(\cT, t)$ and any $\lambda \in \widehat{\kS_n}$, we have $\chi^{\lambda'}(\cT)^t \ch^{\lambda'}(\sigma) = \chi^\lambda(\cT)^t \ch^\lambda(\sigma)$ and $d_{\lambda'} = d_{\lambda}$. Therefore the representations $(n-1,1)$ and $(2, 1^{n-2})$ have the same contribution, which cancels the factor $1/2$. Using that $d_{(n-1,1)} = n-1$ and $\ch^{(n-1,1)}(\sigma) = f_1(\sigma) - 1$ for any $\sigma$, we deduce that
    \begin{equation}
    \begin{split}
         M_1^{(n)}(t) = \sum_{\sigma \in \mathfrak{E}(\cT,t)} \bg \frac{(n-1)}{n!} \chi^\lambda(\cT)^t (f_1(\sigma) - 1)\bd  = \frac{n-1}{2}|\chi^\lambda(\cT)|^t \frac{2}{n!}\sum_{\sigma \in \mathfrak{E}(\cT,t)}|f_1(\sigma) - 1|.
    \end{split}
    \end{equation}
Moreover, we have, letting $X \sim \Unif_{\mathfrak{E}(\cT,t)}$, we have $\bbE[f_1(X)] = 1$ (this uses that $n\geq 3$) so
\begin{equation}
    \frac{2}{n!}\sum_{\sigma \in \mathfrak{E}(\cT,t)}|f_1(\sigma) - 1| = 2\bbP(f_1(X) = 0).
\end{equation}
We therefore obtain
\begin{equation}
       \bg M_1^{(n)}(t)- \frac{(n-1)}{e}\bg \chi^\lambda(\cT) \bd^t  \bd
     = (n-1)\bg \chi^\lambda(\cT)\bd^t \bg \bbP(f_1(X) = 0) - 1/e \bd.
    \end{equation}
Applying Lemma~\ref{lem: borne factorielle permutations paires et impaires sans point fixe} and using that $\bg \chi^\lambda(\cT) \bd = \frac{n-3}{n-1}\leq 1$ concludes the proof.
\end{proof}

\subsection{General approximation}

\begin{proposition}\label{prop: general approximation bound for numerics}
    Let $n\geq 3$ and $t\geq 0$. Then
    \begin{equation}
         \bg \textup{d}^{(n)}(t) - \frac{(n-1)}{e}\pg \frac{n-3}{n-1} \pd^t\bd \leq \varepsilon^{(n)}(t),
    \end{equation}
where 
\begin{equation}
     \varepsilon^{(n)}(t) = \pg \sum_{r=2}^{\lf (n-1)/2\rf} n^r \pg 1 - \frac{2r(n-r+1)}{n(n-1)}\pd^{t}\pd+ \frac{\sqrt{p(n)n!}}{2^{t+1}} + \frac{1}{(n-2)!}.
\end{equation}
\end{proposition}
\begin{proof}
    This follows from \eqref{eq: distance rewritten with M and E}, Lemmas~\ref{lem: bound on E r for r at least 2} and~\ref{lem: factorial approximation for the first representation and its conjugate}, and the triangle inequality.
\end{proof}

\subsection{Preliminaries for numerics}
\begin{lemma}\label{lem: numerics using OEIS tables}
    The following holds.
\begin{enumerate}
    \item $50! \geq 10^{64}$.
    \item $\sqrt{p(52)52!} \leq  5\cdot 10^{36}$.
\end{enumerate}
\end{lemma}
\begin{proof}
The tables of values for the sequences A000142 and A000041 on OEIS give the following exact values (to which we add a space every 10 digits):
\begin{itemize}
    \item $50! = 3041409320\;1713378043\;6126081660\;6476884437\;7641568960\;5120000000\;00000$,
    \item $52! = 8065817517\;0943878571\;6606368564\;0376697528\;9505440883\;2778240000\;00000000$,
    \item $p(52) = 281589$.
\end{itemize}
The results follow.
\end{proof}

\subsection{Numerics for transpositions on a deck of 52 cards}
We give the values of $m_n(t) = \frac{(n-1)}{e}\pg \frac{n-3}{n-1} \pd^{t}$ for $n=52$ and different relevant values of $t$. We note that for $n=52$, $\frac{n-3}{n-1} = \frac{49}{51}$, so $m_n(t+1)$ is about $4\%$ less than $m_n(t)$.
\begin{lemma}\label{lem: numerics for M 1 for 52 cards}
    We have the following approximations.
\begin{enumerate}
    \item $m_{52}(186) \approx  1.101\cdot 10^{-2}$, $m_{52}(187) \approx  1.058\cdot 10^{-2}$, $m_{52}(188) \approx  1.016\cdot 10^{-2}$, $m_{52}(189) \approx  (1- 0.024)\cdot 10^{-2}$, $m_{52}(190) \approx  (1- 0.062)\cdot 10^{-2}$, $m_{52}(191) \approx  (1- 0.099)\cdot 10^{-2}$.
    \item $m_{52}(245) \approx  1.039\cdot 10^{-3}$, $m_{52}(246) \approx  (1-0.002)\cdot 10^{-3}$, $m_{52}(247) \approx  (1-0.041)\cdot 10^{-3}$.
    \item $m_{52}(303) \approx  1.021\cdot 10^{-4}$, $m_{52}(304) \approx  (1- 0.019)\cdot 10^{-4}$.
\end{enumerate}
\end{lemma}
\begin{proof}
    These numerics are obtained by evaluating $m_{52}(t) = \frac{51}{e}\pg \frac{49}{51} \pd^{t}$ for different values of $t$.
\end{proof}

Let us also provide bounds on the error term $\varepsilon^{(n)}(t)$.
\begin{lemma}\label{lem: numerics for the error term for 52 cards}
    Let $n=52$. We have the following bounds.
\begin{enumerate}
    \item For any $t\geq 186$, we have $\varepsilon^{(n)}(t) \leq 9.6\cdot 10^{-4}$. 
     \item For any $t\geq 245$, we have $\varepsilon^{(n)}(t) \leq 10^{-5}$. 
     \item For any $t\geq 303$, we have $\varepsilon^{(n)}(t) \leq 10^{-7}$.
\end{enumerate}
\end{lemma}
\begin{proof}
Since $2^{10} \geq 10^{3}$, for any $t\geq 186$, we have $2^{-(t+1)} \leq 2^{-187} =  8\cdot 2^{-190} \leq 8\cdot 10^{-57}\leq 10^{-56}$. Combining this with Lemma~\ref{lem: numerics using OEIS tables} we get
\begin{equation}
    \frac{\sqrt{p(n)n!}}{2^{t+1}} + \frac{1}{(n-2)!} \leq (5\cdot 10^{36})\cdot 10^{-56} + 10^{-64} \leq 10^{-19},
    \end{equation}
which is so small that we may use it for (a), (b), and (c). For (a) by monotonicity we have that
 for $t\geq 186$,
\begin{equation}
    \sum_{r=2}^{25} n^r \pg 1 - \frac{2r(n-r+1)}{n(n-1)}\pd^{t}\leq \sum_{r=2}^{25} n^r \pg 1 - \frac{2r(n-r+1)}{n(n-1)}\pd^{186}\leq 9.55\cdot 10^{-4},
\end{equation}
where we asked \verb|N[Sum[(52^r)*(1 - 2*r*(52-r+1)/(52*51))^186,{r,2,25}]]| to Wolfram Alpha to obtain the numerical bound. We deduce that for $t\geq 186$ we have $\varepsilon^{(n)}(t) \leq 9.55\cdot 10^{-4} +10^{-19} \leq 9.6\cdot 10^{-4}$, proving (a). We obtain the bounds for (b) and (c) in a similar way.
\end{proof}

\begin{proof}[Proof of Proposition~\ref{prop: numerics for transpositions for n=52}]
    Combining Proposition~\ref{prop: general approximation bound for numerics} with the numerics from Lemmas~\ref{lem: numerics for M 1 for 52 cards} and~\ref{lem: numerics for the error term for 52 cards}, we obtain that $\textup{d}^{(n)}(186)>10^{-2}>\textup{d}^{(n)}(191)$, $\textup{d}^{(n)}(245)>10^{-3}>\textup{d}^{(n)}(247)$, and $\textup{d}^{(n)}(303)>10^{-4}>\textup{d}^{(n)}(304)$. The results follow.
\end{proof}

\section*{Acknowledgements}

We are grateful to Nathanaël Berestycki, Persi Diaconis, and Valentin Féray, for stimulating conversations.
          
\footnotesize{
\bibliographystyle{alpha}
\bibliography{bibliographieLucas}}
\end{document}